\documentclass[11pt]{amsart}

\usepackage{amsmath}
\usepackage{amsthm}
\usepackage{amssymb}
\usepackage{mathrsfs}
\usepackage{verbatim}
\usepackage{esint}
\usepackage[foot]{amsaddr}
\usepackage{subfigure}
\usepackage{tikz}

\usepackage{xcolor}
{%
	\setcounter{enumi}{0}

	\begin{enumerate}}%
	{\end{enumerate} }

{%
	\setcounter{enumi}{0}

	\begin{enumerate}}%
	{\end{enumerate} }

\usepackage[colorlinks=true,linkcolor=blue,citecolor=red]{hyperref}

\usepackage[nameinlink]{cleveref}

\newtheorem{theorem}{Theorem}[section]
\newtheorem{lemma}[theorem]{Lemma}
\newtheorem{proposition}[theorem]{Proposition}

\theoremstyle{definition}
\newtheorem{definition}[theorem]{Definition}
\newtheorem{remark}[theorem]{Remark}

\numberwithin{equation}{section}

\allowdisplaybreaks

\newcommand*\re{\mathbb{R}}
\newcommand*\rn{\mathbb{R}^N}

\newcommand*\dwsp{\dot{W}^{s,p}\left(\rn\right)}

\begin{document}
	
	\title[]{ A system of equations involving the fractional $p$-Laplacian and doubly critical nonlinearities}
	\author{Mousomi Bhakta$^*$}
	\address{${^*}$Department of Mathematics\\
		Indian Institute of Science Education and Research Pune (IISER-Pune)\\
		Dr. Homi Bhabha Road, Pune-411008, India}
	\email{mousomi@iiserpune.ac.in}
	\author{Kanishka Perera$^{**}$}
	\address{$^{**}$Department of Mathematics, Florida Institute of Technology, USA}
	\email{kperera@fit.edu}
	\author{Firoj Sk$^*$}
	\email{firojmaciitk7@gmail.com}
	\smallskip
	
	\maketitle
	\begin{abstract}
		This paper deals with existence of solutions to the following fractional $p$-Laplacian system of equations
		\begin{equation*}
		\begin{cases}
		(-\Delta_p)^s u =|u|^{p^*_s-2}u+ \frac{\gamma\alpha}{p_s^*}|u|^{\alpha-2}u|v|^{\beta}\;\;\text{in}\;\Omega,\\
		(-\Delta_p)^s v =|v|^{p^*_s-2}v+ \frac{\gamma\beta}{p_s^*}|v|^{\beta-2}v|u|^{\alpha}\;\;\text{in}\;\Omega,
		\end{cases}
		\end{equation*}	
		where $s\in(0,1)$, $p\in(1,\infty)$ with $N>sp$, $\alpha,\,\beta>1$ such that $\alpha+\beta = p^*_s:=\frac{Np}{N-sp}$ and $\Omega=\rn$ or smooth bounded domains in $\rn$. When $\Omega=\rn$ and  $\gamma=1$,  we show that any ground state solution of the  above system has the form $(\lambda U, \tau\lambda V)$ for certain  $\tau>0$ and $U,\;V$ are two positive ground state solutions of $(-\Delta_p)^s u =|u|^{p^*_s-2}u$ in $\rn$. For all $\gamma>0$, we establish existence of a positive radial solutions to the above system in  balls. When $\Omega=\rn$, we also establish existence of positive radial solutions to the above system in various ranges of $\gamma$.
	\end{abstract}
	\smallskip

	\section{Introduction}	
	
	We consider the following fractional $p$-Laplacian system of equations in $\rn:$
	
	\begin{equation}
	\tag{$\mathcal S$}\label{MAT1}
	\begin{cases}
	(-\Delta_p)^s u =|u|^{p^*_s-2}u+ \frac{\alpha}{p_s^*}|u|^{\alpha-2}u|v|^{\beta}\;\;\text{in}\;\rn,\\
	(-\Delta_p)^s v =|v|^{p^*_s-2}v+ \frac{\beta}{p_s^*}|v|^{\beta-2}v|u|^{\alpha}\;\;\text{in}\;\rn,
	\\
	u,\;v\in\dot{W}^{s,p}(\rn),
	\end{cases}
	\end{equation}	
	where $0<s<1$, $p\in(1,\infty)$, $N>sp$ and $\alpha,\,\beta>1$ such that $\alpha+\beta = p^*_s:=\frac{Np}{N-sp}$. Here $(-\Delta_p)^s$ denotes the fractional $p$-Laplace operator which can be defined for the Schwartz class functions $\mathcal{S}(\rn)$ as follows
	\begin{equation*} \label{De-u}
	\left(-\Delta_p\right)^su(x):=\text{P.V.} \int_{\rn}\frac{|u(x)-u(y)|^{p-2}\left(u(x)-u(y)\right)}{|x-y|^{N+sp}}dy,\; x\in\rn,
	\end{equation*}
	where P.V. denotes the principle value sense.
	Consider the following homogeneous fractional Sobolev space
	$$
	\dwsp:=\bigg\{u\in L^{p^*_s}(\rn):   \iint_{\mathbb{R}^{2N}}\frac{|u(x)-u(y)|^p}{|x-y|^{N+sp}}dxdy<\infty\bigg\}.
	$$
	The space $\dwsp$ is a Banach space with the corresponding Gagliardo norm $$\|u\|_{\dot{W}^{s,p}(\rn)}:=\bigg(\iint_{\mathbb{R}^{2N}}\frac{|u(x)-u(y)|^p}{|x-y|^{N+sp}}dxdy\bigg)^\frac{1}{p}.$$
	For simiplicity of the notation we write $||u||_{\dot{W}^{s,p}}$ instead of $||u||_{\dwsp}$. In the vectorial case, as described in \cite{BCMP}, the natural solution space for
	\eqref{MAT1} is the product space $X=\dwsp\times\dwsp$ with the norm
	$$\|(u,v)\|_{X}:=\left(\|u\|^2_{\dwsp}+\|v\|^2_{\dwsp}\right)^\frac{1}{2}.$$
	
	\begin{definition}
		We say a pair $(u,v)\in X$ is a positive weak solution of the system \eqref{MAT1} if $u,v>0$ and for every $(\phi,\psi)\in X$ it holds
		\begin{multline*}
		\iint_{\mathbb{R}^{2N}}\frac{|u(x)-u(y)|^{p-2}(u(x)-u(y))(\phi(x)-\phi(y))}{|x-y|^{N+sp}}dxdy\\
		+\iint_{\mathbb{R}^{2N}}\frac{|v(x)-v(y)|^{p-2}(v(x)-v(y))(\psi(x)-\psi(y))}{|x-y|^{N+sp}}dxdy\\
		=\int_{\rn}|u|^{p^*_s-2}u\phi\;dx+\int_{\rn}|v|^{p^*_s-2}v\psi\;dx
		+\frac{\alpha}{p^*_s}\int_{\rn}|u|^{\alpha-2}u|v|^{\beta}\phi\;dx+\frac{\beta}{p^*_s}\int_{\rn}|v|^{\beta-2}v|u|^{\alpha}\psi\;dx.
		\end{multline*}
	\end{definition}
	Define
	\begin{equation}\label{best constant alpha+beta}
	\mathcal{S}=S_{\alpha+\beta}:=\inf_{\substack{u\in\dot{W}^{s,p}(\rn),\\
			u\neq0}}\frac{\|u\|_{\dot{W}^{s,p}}^p}{\left(\displaystyle\int_{\rn}|u|^{p^*_s}dx\right)^{\frac{p}{p^*_s}}}.
	\end{equation}
	In the limit case $p=1$, the sharp constant $\mathcal{S}$ has been determined in \cite[Theorem 4.1]{FS} (see also \cite[Theorem 4.10]{BLP}). The relevant extremals are given by the characteristic functions of
	balls, exactly as in the local case. For $p>1$, \eqref{best constant alpha+beta} is related to the study of the following nonlocal integro-differential
	equation
	\begin{equation}\label{scalar equation}
	\begin{cases}
	(-\Delta_p)^su=\mathcal S \;u^{p^*_s-1} \text{ in }\rn,\\
	u>0, \quad  u\in\dot{W}^{s,p}(\rn).
	\end{cases}
	\end{equation}
	In the Hilbertian case $p=2$, it is known by \cite[Theorem 1.1]{CT}, the best Sobolev constant $\mathcal{S}$ is attained by the family of functions
	$$U_t(x)=t^\frac{2s-N}{2}\bigg(1+\big(\frac{|x-x_0|}{t}\big)^2\bigg)^\frac{2s-N}{2}, \quad x_0\in\rn,\quad  t>0. $$
	Moreover, the family $U_t$ is the only set of minimizers for the best Sobolev constant (see \cite{CLO}). However, for $p\neq 2$, the minimizers of $\mathcal{S}$ is not yet known and it is not known whether \eqref{best constant alpha+beta} has any unique minimizer or not. In \cite{BMS}, Brasco, et. al. have conjectured that the optimizers of $\mathcal{S}$ in \eqref{best constant alpha+beta} are given by
	$$U_t(x)=Ct^\frac{sp-N}{p}\bigg(1+\big(\frac{|x-x_0|}{t}\big)^\frac{p}{p-1}\bigg)^\frac{sp-N}{p}, \quad x_0\in\rn,\,  t>0, $$
	but it remains as an open question till date.
	However, in \cite[Theorem 1.1]{BMS},  it has been proved that if $U$ is any minimizer of $\mathcal{S}$ then $U$ is of
	constant sign, radially symmetric and monotone function with
	$$\lim_{|x|\to\infty}|x|^\frac{N-sp}{p-1}U(x)=U_\infty,$$
	for some constant $U_\infty\in\mathbb{R}\setminus\{0\}$.

	\medskip
	Peng et. al in \cite{PPW} studied system \eqref{MAT1} for $p=2$ and $s=1$ and among the other results they proved uniqueness of least energy solution. In the local case $s=1$, a variant of system \eqref{MAT1} ( with $p=2$ ) appears in various context of mathematical physics e.g. in Bose-Einstein condensates theory, nonlinear wave-wave interaction in plasma physics, nonlinear optics, for details see \cite{AA, Bh, PW} and the references therein.  System of elliptic $p-$ Laplacian type equations with weakly coupled nonlinearities we also cite \cite{FP, GuPeZo} and the references therein. In the nonlocal case, there are not so many papers, in which weakly coupled systems of equations have been studied. We refer to \cite{CS, CMSY, FMPSZ, HSZ}, where Dirichlet systems of
	equations in bounded domains have been treated. For the nonlocal systems of equations in the
	entire space $\rn$, we cite \cite{BCP, FPS, FPZ} and the references therein.

	For $p=2$ and $s\in(0,1)$ Bhakta, et al. in \cite{BCMP} studied the following system:
	\begin{equation}
	\label{S2}
	\begin{cases}
	(-\Delta)^s u = \frac{\alpha}{2_s^*}|u|^{\alpha-2}u|v|^{\beta}+f(x)\;\;\text{in}\;\rn,\\
	(-\Delta)^s v = \frac{\beta}{2_s^*}|v|^{\beta-2}v|u|^{\alpha}+g(x)\;\;\text{in}\;\rn,
	\\
	u,\;v>0   \quad\text{ in }\rn,
	\end{cases}
	\end{equation}	
	where $f, g$ belongs to the dual space of $\dot{W}^{s,2}(\rn)$. Among other results the authors proved that when $f=0=g$, any ground state solution of \eqref{S2} has the form $(Bw, Cw)$, where $C/B=\sqrt{\beta/\alpha}$ and $w$  is the unique solution of \eqref{scalar equation} (corresponding to $p=2$).
	
	\medskip
	
	Being inspired by the above works, in this paper we generalize some of the above results in the fractional-$p$-Laplacian case.
	
	\begin{definition}
		(i) We say a weak solution $(u,v)$ of \eqref{MAT1} is of the synchronized form if $u=\lambda w,$ $v=\mu w$ for some constants $\lambda, \mu$ and a common function $w\in\dot{W}^{s,p}(\mathbb{R}^N).$
		
		(ii) We say a weak solution $(u,v)$ of \eqref{MAT1} is a ground state solution if $(u,v)$ is a minimizer of $S_{\alpha, \beta}$ (see below \eqref{best constant alpha, beta}). 
	\end{definition}
	Define,
	\begin{equation}\label{best constant alpha, beta}
	S_{\alpha,\beta}:=\inf_{\substack{(u,v)\in X,\\
			(u,v)\neq0}}\frac{\|u\|_{\dot{W}^{s,p}}^p+\|v\|_{\dot{W}^{s,p}}^p}{\left(\displaystyle\int_{\rn}\left(|u|^{p^*_s}+|v|^{p^*_s}+|u|^{\alpha}|v|^{\beta}\right)dx\right)^{\frac{p}{p^*_s}}}.
	\end{equation}
	
	Suppose that \eqref{MAT1} has a positive solution of the synchronized form $(\lambda U, \mu U)$ for some $\lambda>0$, $\mu>0$ and $U\in \dwsp$ is a ground state solution of \eqref{scalar equation}. Then it holds
	\begin{equation*}
	\lambda^{p^*_s-p}+\frac{\alpha}{p^*_s}\mu^{\beta}\lambda^{\alpha-p}=1= \mu^{p^*_s-p}+\frac{\beta}{p^*_s}\mu^{\beta-p}\lambda^{\alpha}.
	\end{equation*}
	Now setting $\mu=\tau\lambda$, we get $\lambda^{p^*_s-p}=\frac{p^*_s}{p^*_s+\alpha\tau^{\beta}}$ and $\tau$ satisfies
	\begin{equation}\label{eq:tau} p^*_s+\alpha\tau^{\beta}-\beta\tau^{\beta-p}-p^*_s\tau^{p^*_s-p}=0.\end{equation}
	On the other hand, we find that, if $\tau$ satisfies \eqref{eq:tau}, then $(\lambda U, \tau\lambda U)$ solves \eqref{MAT1}.
	
	Therefore, the natural question arises: is all the ground state solutions of \eqref{MAT1} of the
	synchronized form $(\lambda U, \tau\lambda U)$?
	
	If the answer of the above question is affirmative then it will hold
	$$S_{\alpha,\beta}=\frac{1+\tau^p}{\left(1+\tau^\beta+\tau^{p^*_s}\right)^{p/{p^*_s}}}\mathcal{S}.$$
	This inspires us to define the following function
	\begin{equation}\label{eq:h}
	h(\tau):=\frac{1+\tau^p}{\left(1+\tau^\beta+\tau^{p^*_s}\right)^{p/{p^*_s}}}.\end{equation}
	Note that $h(\tau_{min})=\min_{\tau\geq 0}h(\tau)\leq 1$.
	
	\medskip
	
	Below we state the main results of this paper.
	
	\begin{theorem}\label{main-gr1}
		Let $(u_0,v_0)$ be any positive ground state solution of \eqref{MAT1}. If one of the following conditions hold
		
		\begin{enumerate}
			\item [(i)] $1<\beta<p$,
			\item [(ii)] $\beta=p$ and $\alpha<p$,
			\item [(iii)] $\beta>p$ and  $\alpha<p$,
		\end{enumerate}
		then, there exists unique $\tau_{min}>0$ satisfying $$h(\tau_{min})=\min_{\tau\geq 0}h(\tau)< 1,$$ where $h$ is defined by \eqref{eq:h}. Moreover,
		$$
		(u_0,v_0)=\left(\lambda U, \tau_{min}\lambda V\right),
		$$
		where $U,\;V$ are two positive ground state solutions of \eqref{scalar equation}.
		Further,  $\lambda^{p^*_s-p}=\frac{p^*_s}{p^*_s+\alpha\tau_{min}^{\beta}}$. \end{theorem}
	
	\begin{remark}
		Since for $p\neq 2$, uniqueness of ground state solutions of \eqref{scalar equation} is not yet known, we are not able to conclude whether any ground state solution of \eqref{MAT1} is of the synchronized form i.e, of the form of $(\lambda U, \tau_{min}\lambda U)$ or not.
	\end{remark}

	Next, we consider \eqref{MAT1} with a small perturbation $\gamma>0$, namely we consider the system
	
	\begin{equation}
	\tag{$\tilde{\mathcal S_{\gamma}}$}\label{MAT2}
	\begin{cases}
	(-\Delta_p)^s u =|u|^{p^*_s-2}u+ \frac{\alpha\gamma}{p_s^*}|u|^{\alpha-2}u|v|^{\beta}\;\;\text{in}\;\rn,\vspace{.2cm}\\
	(-\Delta_p)^s v =|v|^{p^*_s-2}v+ \frac{\beta\gamma}{p_s^*}|v|^{\beta-2}v|u|^{\alpha}\;\;\text{in}\;\rn,
	\\
	u,\;v\in\dot{W}^{s,p}(\rn).
	\end{cases}
	\end{equation}
	and prove existence of positive solutions to \eqref{MAT2} in various range of $\gamma$. The corresponding energy functional of the problem \eqref{MAT2}, given by for $(u,v)\in X$
	
	\begin{equation}\label{J-functional}
	\mathcal J(u,v)=\frac{1}{p}\left(\|u\|_{\dot{W}^{s,p}}^p+\|v\|_{\dot{W}^{s,p}}^p\right)-\frac{1}{p^*_s}\int_{\rn}\left(|u|^{p^*_s}+|v|^{p^*_s}+\gamma|u|^{\alpha}|v|^{\beta}\right)dx.
	\end{equation}
	We define
	
	\begin{multline}\label{Nehari}
	\mathcal N=\left\{(u,v)\in X:u\neq0,\;v\neq0,\;\|u\|_{\dot{W}^{s,p}}^p=\int_{\rn}\left(|u|^{p^*_s}+\frac{ \alpha\gamma}{p^*_s}|u|^{\alpha}|v|^{\beta}\right)dx\right.,\\
	\left. \|v\|_{\dot{W}^{s,p}}^p=\int_{\rn}\left(|v|^{p^*_s}+\frac{ \beta\gamma}{p^*_s}|u|^{\alpha}|v|^{\beta}\right)dx\right\}.
	\end{multline}
	It is easy to see that $\mathcal N\neq\emptyset$ and that any nontrivial solution of \eqref{MAT2} is belongs to $\mathcal N.$ Set
	$$A:=\inf\limits_{(u,v)\in\mathcal N}\mathcal{J}(u,v).$$
	Consider the nonlinear system of algebraic equations
	\begin{equation}\label{algebraic eq1}
	\begin{cases}
	k^{\frac{p^*_s-p}{p}}+\frac{\alpha\gamma}{p^*_s}k^{\frac{\alpha-p}{p}}\ell^{\frac{\beta}{p}}=1,\\
	\ell^{\frac{p^*_s-p}{p}}+\frac{\beta\gamma}{p^*_s}\ell^{\frac{\beta-p}{p}}k^{\frac{\alpha}{p}}=1,\\
	k,\ell>0.
	\end{cases}
	\end{equation}
	\begin{theorem}\label{minimizer value}
		Assume that one of the following conditions hold:
		\begin{itemize}
			\item [i)] If $\frac{N}{2s}<p<\frac{N}{s}$, $\alpha,\, \beta>p$ and
			\begin{equation}\label{condition on gamma: 1}
			0<\gamma\leq\frac{p^*_s(p^*_s-p)}{p}\min\left\{\frac{1}{\alpha}\left(\frac{\alpha-p}{\beta-p}\right)^{\frac{\beta-p}{p}},\;\frac{1}{\beta}\left(\frac{\beta-p}{\alpha-p}\right)^{\frac{\alpha-p}{p}}\right\};
			\end{equation}
			\smallskip
			\item[ii)] If $\frac{2N}{N+2s}<p<\frac{N}{2s}$, $\alpha,\,\beta<p$ and
			\begin{equation}\label{condition on gamma: 2}
			\gamma\geq\frac{p^*_s(p^*_s-p)}{p}\max\left\{\frac{1}{\alpha}\left(\frac{p-\beta}{p-\alpha}\right)^{\frac{p-\beta}{p}},\;\frac{1}{\beta}\left(\frac{p-\alpha}{p-\beta}\right)^{\frac{p-\alpha}{p}}\right\}.
			\end{equation}
		\end{itemize}
		Then the least energy $A=\frac{s}{N}(k_0+\ell_0)\mathcal S^{N/sp} $ and $A$ is attained by $\left(k_0^{1/p}U,\ell_0^{1/p}U\right)$, where $U$ is a minimizer of \eqref{best constant alpha+beta}, $k_0,\;\ell_0$ satisfies \eqref{algebraic eq1} and
		\begin{equation}\label{condition on k_0}
		k_0=\min\{k:(k,\ell)\text{ satisfies }\eqref{algebraic eq1}\}
		\end{equation}
	\end{theorem}
	
	\begin{theorem}\label{gamma near 0}
		Assume that $\frac{2N}{N+2s}<p<\frac{N}{2s}$ and $\alpha,\,\beta<p$. There exists $\gamma_1>0$
		such that for any $\gamma\in(0,\gamma_1)$ there exists a solution $(k(\gamma),\ell(\gamma))$ of \eqref{algebraic eq1} such that
		$\left(k(\gamma)^{1/p}U,\ell(\gamma)^{1/p}U\right)$ is a positive solution of system \eqref{MAT2} with $
		\mathcal{J}\left(k(\gamma)^{1/p}U,\ell(\gamma)^{1/p}U\right)>\tilde{A}
		$, where $U$ is a minimizer of \eqref{best constant alpha+beta},
		$$\tilde{A}=\inf_{(u,v)\in\mathcal{\tilde N}}\mathcal{J}(u, v)$$ and
		$$\tilde{\mathcal N}=\left\{(u,v)\in X\setminus\{0\}:\|u\|_{\dot{W}^{s,p}}^p+\|v\|_{\dot{W}^{s,p}}^p=\int_{\rn}\left(|u|^{p^*_s}+|v|^{p^*_s}+\gamma|u|^{\alpha}|v|^{\beta}\right)dx\right\}.$$
		
	\end{theorem}
	
	\begin{theorem}\label{in ball}
		Assume that $\frac{2N}{N+2s}<p<\frac{N}{2s}$ and $\alpha,\beta<p.$ Then the following system of equations
		\begin{equation}\label{system_ball}
		\begin{cases}
		(-\Delta_p)^s u =|u|^{p^*_s-2}u+ \frac{\alpha\gamma}{p_s^*}|u|^{\alpha-2}u|v|^{\beta}\;\;\text{in}\;B_R(0),\vspace{.2cm}
		\\
		(-\Delta_p)^s v =|v|^{p^*_s-2}v+ \frac{\beta\gamma}{p_s^*}|v|^{\beta-2}v|u|^{\alpha}\;\;\text{in}\;B_R(0),
		\\
		u,\;v\in W^{s,p}_{0}(B_R(0)),
		\end{cases}
		\end{equation}
		admit a radial positive solution $(u_0,v_0)$.
	\end{theorem}
	
	The organization of the rest of the paper is as follows: In Section 2, we prove Theorem~\ref{main-gr1}. Section 3 deals with proof of Theorem~\ref{minimizer value}, \ref{in ball} and \ref{gamma near 0} respectively.

	\section{Proof of Theorem~\ref{main-gr1}}
	
	\begin{lemma}\label{basic lemma}
		Suppose $\alpha,\;\beta>1$ such that $\alpha+\beta=p^*_s.$ Then
		\begin{itemize}
			\item[i)] $S_{\alpha,\beta}=h(\tau_{min})\mathcal{S}. $ 	\item[ii)]	$S_{\alpha,\beta}$ has minimizers $(U,\tau_{min}U)$, where $U$ is a ground state solution of \eqref{scalar equation} and $\tau_{min}$ satisfies

			$$
			\tau^{p-1}\left(p^*_s+\alpha\tau^{\beta}-\beta\tau^{\beta-p}-p^*_s\tau^{p^*_s-p}\right)=0.
			$$
		\end{itemize}
	\end{lemma}
	\begin{proof}
		Let $\{(u_n, v_n)\}$ be a minimizing sequence in $X$ for $S_{\alpha,\beta}$. Choose $\tau_n>0$ such that $\|v_n\|_{L^{p^*_s}(\rn)}=\tau_n\|u_n\|_{L^{p^*_s}(\rn)}$. Now set, $w_n=\frac{v_n}{\tau_n}$. Therefore, $\|u_n\|_{L^{p^*_s}(\rn)}=\|w_n\|_{L^{p^*_s}(\rn)}$ and applying Young's inequality,
		$$\int_{\rn}|u_n|^\alpha|w_n|^{\beta}dx\leq \frac{\alpha}{p^*_s}\int_{\rn}|u_n|^{p^*_s}dx+\frac{\beta}{p^*_s}\int_{\rn}|w_n|^{p^*_s}dx=\int_{\rn}|u_n|^{p^*_s}dx=\int_{\rn}|w_n|^{p^*_s}dx.$$
		Therefore,
		\begin{multline*}
		S_{\alpha,\beta}+o(1)
		=\frac{\|u_n\|_{\dot{W}^{s,p}}^p+\|v_n\|_{\dot{W}^{s,p}}^p}{\left(\displaystyle\int_{\rn}\left(|u_n|^{p^*_s}+|v_n|^{p^*_s}+|u_n|^{\alpha}|v_n|^{\beta}\right)dx\right)^{\frac{p}{p^*_s}}}
		\\
		=\frac{\|u_n\|_{\dot{W}^{s,p}}^p}{\left(\displaystyle\int_{\rn}\left(|u_n|^{p^*_s}+\tau_n^{p^*_s}|u_n|^{p^*_s}+\tau_n^{\beta}|u_n|^{\alpha}|w_n|^{\beta}\right)dx\right)^{\frac{p}{p^*_s}}}
		\\
		+\frac{\tau_n^p\;\|w_n\|_{\dot{W}^{s,p}}^p}{\left(\displaystyle\int_{\rn}\left(|u_n|^{p^*_s}+\tau_n^{p^*_s}|u_n|^{p^*_s}+\tau_n^{\beta}|u_n|^{\alpha}|w_n|^{\beta}\right)dx\right)^{\frac{p}{p^*_s}}}
		\\
		\geq\frac{1}{\left(1+\tau_n^\beta+\tau_n^{p^*_s}\right)^{p/{p^*_s}}}\left[\frac{\|u_n\|_{\dot{W}^{s,p}}^p}{\left(\displaystyle\int_{\rn}|u_n|^{p^*_s}dx\right)^{\frac{p}{p^*_s}}}+\frac{\tau_n^p\;\|w_n\|_{\dot{W}^{s,p}}^p}{\left(\displaystyle\int_{\rn}|w_n|^{p^*_s}dx\right)^{\frac{p}{p^*_s}}}\right]
		\\
		\geq\frac{1+\tau_n^p}{\left(1+\tau_n^\beta+\tau_n^{p^*_s}\right)^{p/{p^*_s}}}\mathcal{S}\geq\min_{\substack{\tau>0}}h(\tau)\mathcal{S}.
		\end{multline*}	
		Thus, as $n\to\infty$, we have $h(\tau_{min})\mathcal{S}\leq S_{\alpha,\beta}.$	
		For the reverse inequality, we choose $u=U$, $v=\tau_{min}U$ to get
		$h(\tau_{min})\mathcal{S}\geq S_{\alpha,\beta}.$ In Lemma~\ref{l:g function}, we will show that point $\tau_{min}$ exists.   This proves i).
		
		ii) Taking $(u,v)=(U,\tau_{min}U)$, a simple computation yields that
		$$
		\frac{\|u\|_{\dot{W}^{s,p}}^p+\|v\|_{\dot{W}^{s,p}}^p}{\left(\displaystyle\int_{\rn}\left(|u|^{p^*_s}+|v|^{p^*_s}+|u|^{\alpha}|v|^{\beta}\right)dx\right)^{\frac{p}{p^*_s}}}=h(\tau_{min})\mathcal{S}.
		$$	
		By using i), we infer that $(U,\tau_{min}U)$ is a minimizer of $S_{\alpha,\beta}$. Further, since $\tau_{min}$ is a critical point of $h$, computing $h'(\tau_{min})=0$ yields that $\tau_{min}$ satisfies
		$$\tau^{p-1}\left(p^*_s+\alpha\tau^{\beta}-\beta\tau^{\beta-p}-p^*_s\tau^{p^*_s-p}\right)=0.$$
	\end{proof}
	
	We observe from \eqref{eq:h} that $h(0)=1$ and $\lim_{\tau\to\infty}h(\tau)=1$. Therefore, to ensure the existence of $\tau_{min}$ (i.e., minimum point of $h$ does not escape at infinity), $\tau_{min}$ is uniquely defined and $\tau_{min}>0$, we need to investigate the solvability of the following equation
	\begin{equation}\label{def: g function}
	g(\tau):=p^*_s+\alpha\tau^{\beta}-\beta\tau^{\beta-p}-p^*_s\tau^{p^*_s-p}=0.
	\end{equation}
	\begin{lemma}\label{l:g function}
		Let $\alpha,\;\beta>1$ and $\alpha+\beta=p^*_s$. Then \eqref{def: g function}  always has atleast one root $\tau>0$ and for any root $\tau>0$, the problem \eqref{MAT1} has positive solutions $(\lambda U,\mu U)$, where
		$$
		\mu=\tau\lambda,\quad \lambda^{p^*_s-p}=\frac{p^*_s}{p^*_s+\alpha\tau^{\beta}}.
		$$
		Moreover, if one of the following conditions hold
		\begin{enumerate}
			\item [(i)] $1<\beta<p$,
			\item [(ii)] $\beta=p$ and $\alpha<p$,
			\item [(iii)] $\beta>p$ and  $\alpha<p$,
		\end{enumerate}
		then,  $\tau_{min}>0$ and $h(\tau_{min})<1$. In all other cases, $\tau_{min}=0$.
	\end{lemma}
	
	\begin{proof}
		Clearly, if $\tau>0$ solves
		\begin{equation*}
		\begin{cases}
		\left(p^*_s+\alpha\tau^{\beta}\right)\lambda^{p^*_s-p}=p^*_s,
		\\
		\left(p^*_s\tau^{p^*_s-p}+\beta\tau^{\beta-p}\right)\lambda^{p^*_s-p}=p^*_s,
		\end{cases}
		\end{equation*}
		then $(\lambda U,\mu U)$ with $\mu=\tau\lambda$ solves \eqref{MAT1}. Thus to prove the required result, it is enough to show that \eqref{def: g function} has positive roots $\tau,$ which we discuss in the following cases.
		\medskip
		
		\noindent\textbf{\textit{Case 1:}} If $1<\beta<p.$
		\newline Therefore, $\lim\limits_{\tau\to0^+}g(\tau)=-\infty.$
		\newline Now, if $\alpha\geq p,$ then $g(1)=\alpha-\beta>0$. Thus, there exists $\tau\in(0, 1)$ such that $g(\tau)=0.$\\
		If $1<\alpha<p$, then we have $p^*_s-p<p^*_s-\alpha=\beta$, and consequently, $\lim\limits_{\tau\to\infty}g(\tau)=\infty.$ Thus there exists $\tau>0$ such that $g(\tau)=0.$
		
		Also observe that, by direct computation we obtain
		$$h'(\tau)=f(\tau)g(\tau), \quad\text{where}\quad f(\tau)=\frac{p\tau^{p-1}}{p^*_s(1+\tau^\beta+\tau^{p^*_s})^{\frac{p}{p^*_s}+1}}.$$
		Thus, $f(\tau)\geq 0$ for all $\tau>0$ and $f(0)=0$. This together with the fact that $\lim\limits_{\tau\to0^+}g(\tau)=-\infty$ implies $h'(\tau)<0$ in $\tau\in(0,\epsilon)$
		for some $\epsilon>0$. This means $h$ is a decreasing function near $0$. Combining this with the fact that $h(0)=1$ and $\lim_{\tau\to\infty}h(\tau)=1$, we conclude that there exists a point $\tau_{min}\in (0,\infty)$ such that
		$\min_{\tau\geq 0}h(\tau)=h(\tau_{min})<1$ and this holds for all $\alpha>1$.
		
		\noindent\textbf{\textit{Case 2:}} If $\beta=p.$
		\newline In this case $g$ becomes $g(\tau)=\alpha(1+\tau^p)-p^*_s\tau^{\alpha}$. Hence $g(0)=\alpha>0$ and $g(1)=\alpha-p.$\\
		(i) If $\alpha=p$, then $N=2sp$ and $g(\tau)=p-p\tau^p$. Thus there exists a unique root $\tau_1=1$ of $g$. Also note that $h$ is increasing near 0. Hence $\tau_1$ is the maximum point of $h$ with $h(\tau_1)>1$. In this case $\min_{\tau\geq 0}h(\tau)=h(\tau_{min})=h(0)$.
		\begin{figure}[h]
			\centering
			\subfigure[The case $2(i)$] {
				\begin{tikzpicture}[scale=.8]
				\draw[->][] (-.25,0)--(7,0) node[anchor=west] {\tiny{$\tau$}};
				\draw[->][] (0,-.25)--(0,2) node[anchor=east] {\tiny{$y$}};
				\draw[](6,1)--(0,1) node[anchor=east] {\tiny{$y=1$}};

				\draw[][-] [domain=0:5.5,smooth] plot ({\x}, {10*((1+\x^1.5)/(1+\x^1.5+\x^3)^.5)-9});
				\draw[](4,2.15) node[anchor=north east ]{\tiny{$h(\tau)$}};
				\draw[](.99,2.53)--(.99,0) node[anchor=north]{\tiny{$\tau_1$}};
				\end{tikzpicture}
			}
		\end{figure}\\
		(ii) If $1<\alpha< p$, then we have $2sp<N<sp(p+1)$. Observe that $g(0)=\alpha>0$, $g(1)=\alpha-p<0$ and $\lim\limits_{\tau\to\infty}g(\tau)=+\infty$. Also note that, $g$ is decreasing in $\left(0,\left(p^*_s/p\right)^{\frac{1}{p-\alpha}}\right)$ and increasing in $\left(\left(p^*_s/p\right)^{\frac{1}{p-\alpha}},\infty\right)$. Therefore, $g$ has exactly one critical point $\left(p^*_s/p\right)^{\frac{1}{p-\alpha}}$ and two roots $\tau_i \;(i=1,2)$ with $\tau_1\in(0, \left(p^*_s/p\right)^{\frac{1}{p-\alpha}})$, $\tau_2\in(\left(p^*_s/p\right)^{\frac{1}{p-\alpha}},\infty)$.
		\begin{figure}[h]
			\subfigure[The case $2(ii)$] {
				\begin{tikzpicture}[scale=.7]
				\draw[->][] (-.25,0)--(4,0) node[anchor=west] {\tiny{$\tau$}};
				\draw[->][] (0,-.25)--(0,2) node[anchor=east] {\tiny{$y$}};

				\draw[][-] [domain=0:3.5] plot ({\x}, {1.1+1.1*\x^2-3.1*\x^1.1});
				\draw[](3.35,2) node[anchor=north east ]{\tiny{$g(\tau)$}};
				\draw[](1,0) node[anchor=south east ]{\tiny{$\tau_1$}};
				\draw[](2.9,0) node[anchor=south east ]{\tiny{$\tau_2$}};
				\end{tikzpicture}

				\begin{tikzpicture}[scale=.85]
				\draw[->][] (-.25,0)--(7,0) node[anchor=west] {\tiny{$\tau$}};
				\draw[->][] (0,-.25)--(0,2) node[anchor=east] {\tiny{$y$}};
				\draw[](7,1)--(0,1) node[anchor=east] {\tiny{$h(0)=1$}};
				
				\draw[][-] [domain=0:6,smooth] plot ({\x}, {8*((1+\x^2.5)/(1+\x^2.5+\x^3.9)^.64)-7});
				\draw[](5,.68) node[anchor=north east ]{\tiny{$h(\tau)$}};
				\draw[](.55,1.2)--(.55,0) node[anchor=north]{\tiny{$\tau_1$}};
				\draw[](2,.47)--(2,0) node[anchor=north]{\tiny{$\tau_2$}};
				\end{tikzpicture}
			}
		\end{figure}
		Further, note that in this case $h$ is increasing near $0$ which leads that first positive critical point of $h$, i.e., $\tau_1$ is the local maximum for $h$ and $h(\tau_1)>1$. Further, as $\lim\limits_{\tau\to\infty}h(\tau)=1$, the second root of $g$ i.e.,  $\tau_2$ becomes the 2nd and last critical point of $h$ and $h(\tau_2)<1.$ Therefore, in this case $\tau_{min}=\tau_2>0$ is the minimum point of $h$ with $h(\tau_{min})<1.$
		\\
		(iii) If $\alpha>p$, then $sp<N<2sp$ and $g(0)>0$. We see that $g$ is increasing in $\left(0,\left(p/p^*_s\right)^{\frac{1}{\alpha-p}}\right)$ and decreasing in $\left(\left(p/p^*_s\right)^{\frac{1}{\alpha-p}},\infty\right)$. This together with the fact  $\lim\limits_{\tau\to\infty}g(\tau)=-\infty$ leads that there exists a unique $\tau>0$ such that $g(\tau)=0$.
		\begin{figure}[h]
			\subfigure[The case $2(iii)$] {
				\begin{tikzpicture}[scale=.8]
				\draw[->][] (-.25,0)--(4,0) node[anchor=west] {\tiny{$\tau$}};
				\draw[->][] (0,-.25)--(0,3) node[anchor=east] {\tiny{$y$}};

				\draw[][-] [domain=0:1.2] plot ({\x}, {2+2*\x^1.1-3.1*\x^3});
				\draw[](2,1) node[anchor=north east ]{\tiny{$g(\tau)$}};
				\end{tikzpicture}
				
				\begin{tikzpicture}[scale=.8]
				\draw[->][] (-.25,0)--(7,0) node[anchor=west] {\tiny{$\tau$}};
				\draw[->][] (0,-.25)--(0,2) node[anchor=east] {\tiny{$y$}};
				\draw[](6,1)--(0,1) node[anchor=east] {\tiny{$y=1$}};

				\draw[][-] [domain=0:5,smooth] plot ({\x}, {10*((1+\x^1.5)/(1+\x^1.5+\x^3)^.5)-9});
				\draw[](4,2.15) node[anchor=north east ]{\tiny{$h(\tau)$}};
				\draw[](.99,2.53)--(.99,0) node[anchor=north]{\tiny{$\tau_1$}};
				\end{tikzpicture}
			}
		\end{figure}
		
		Since in this case, $h$ is increasing near $0$, so at $\tau$, $h$ attains the maximum with $h(\tau)>1.$ Hence, $h$ has no other critical point and therefore, $h(\tau_{min})=h(0)$.
		
		\medskip
		
		\noindent\textbf{\textit{Case 3:}} If $\beta>p.$
		
		If $1<\alpha\leq p$, then $g(1)=\alpha-\beta\leq 0$. Since $g(0)>0$, there is a $\tau\in(0,1]$ such that $g(\tau)=0.$ If $\alpha>p$ and $\alpha>\beta$, then $g(1)>0$ and $\lim\limits_{\tau\to\infty}g(\tau)=-\infty$. Thus there exists $\tau\in(1,\infty)$ such that $g(\tau)=0.$ If $\alpha>p$ and $\alpha\leq\beta$, then $g(1)\leq0$. As $g(0)>0$, thus there exists $\tau\in(0,1]$ such that $g(\tau)=0.$
		Next we analyse $\tau_{min}$ in case 3 in the following three subcases.
		
		\smallskip
		
		(i) $\beta>p$ and $\alpha>p$.
		
		Observe that in this case we have
		\begin{equation}\label{alb-1}\beta<p^*_s-p \quad\text{and}\quad \alpha<p^*_s-p.\end{equation}
		Hence, without loss of generality we can assume $\alpha\geq\beta$.
		
		Claim 1: $g(\tau)>0$ for $\tau\in[0,1)$. Indeed, using \eqref{alb-1}, $\tau\in[0,1)$ implies $\tau^\alpha, \tau^{\beta}>\tau^{p^*_s-p}$. Therefore,
		\begin{align*}
		g(\tau)&>p^*_s+\alpha\tau^{\beta}-\beta\tau^{\beta-p}-p^*_s\tau^{\beta}\\
		&=p^*_s+(\alpha-p^*_s)\tau^\beta-\beta\tau^{\beta-p}\\
		&>p^*+\alpha-p^*_s-\beta\tau^{\beta-p} \quad \text{(as $\alpha<p^*_s$ and $\tau^\beta<1$)}\\
		&=\alpha-\beta\tau^{\beta-p}>0,
		\end{align*}
		where in the last inequality we have used the fact that $\tau^{\beta-p}<1\Longrightarrow  \tau^{\beta-p}<\beta\leq\alpha$. This proves the claim 1.
		
		Claim 2: $g$ is monotonically decreasing for $\tau\geq 1$. Indeed, $\tau\geq 1$ implies $\tau^\alpha\geq\tau^p$. Therefore, using \eqref{alb-1} we have
		\begin{align*}
		g^{\prime}(\tau)&=\tau^{\beta-p-1}\big[\alpha\beta\tau^{p}-p^*_s(p^*_s-p)\tau^{\alpha}-\beta(\beta-p)\big]\\
		&\leq\tau^{\beta-p-1}\big[(\alpha\beta-p^*_s(p^*_s-p))\tau^\alpha-\beta(\beta-p)\big]\\
		&\leq \tau^{\beta-p-1}\big[(p^*_s-p)(\beta-p^*_s)\tau^\alpha-\beta(\beta-p)\big]\\
		&<0.
		\end{align*}
		This proves claim 2. Also observe that $g(1)\geq 0$ and $g(\tau)\to -\infty$ as $\tau\to\infty$. Combining these facts along with claim 1 and 2 above proves that $g$ has only one root say $\tau$ in $(0,\infty)$, which in turn implies $h$ has only one critical point $\tau$ in $(0,\infty)$. Since $\beta>p$ implies $h$ is increasing near $0$, so at $\tau$, $h$ attains the maximum with $h(\tau)>1$. Combining this with $\lim_{\tau\to\infty}h(\tau)=1$ proves that $h(\tau_{min})=h(0)=1$, i.e, $\tau_{min}=0$.
		
		\medskip
		
		(ii) $\beta>p$ and $\alpha<p$.
		
		In this case $g(0)>0$, $g(1)<0$ and we claim $g$ is strictly decreasing in $(0,1)$. Indeed, $\alpha<p \Longrightarrow \tau^p<\tau^\alpha$ for $\tau\in(0,1)$. Also $\beta>p\Longrightarrow \alpha<p^*_s-p$. Therefore,
		\begin{align*}
		g^{\prime}(\tau)&=\tau^{\beta-p-1}\big[\alpha\beta\tau^{p}-p^*_s(p^*_s-p)\tau^{\alpha}-\beta(\beta-p)\big]\\
		&<\tau^{\beta-p-1}\big[(p^*_s-p)(\beta-p^*_s)\tau^\alpha-\beta(\beta-p)\big]\\
		&<0.
		\end{align*}
		Claim: $g$ has only one critical point in $(1,\infty)$.
		Indeed,
		$$g'(\tau)=\tau^{\beta-p-1}g_1(\tau), \quad\text{where}\quad g_1(\tau):=\alpha\beta\tau^{p}-p^*_s(p^*_s-p)\tau^{\alpha}-\beta(\beta-p).$$
		So to prove that $g$ has only critical point in $(1,\infty)$, it's enough to show that
		$g_1$ has only one root in $(1,\infty)$. Observe that, $g_1(0)<0$, $\lim_{\tau\to\infty}g_1(\tau)=\infty$ and a straight forward computation yields that $g_1$ is a decreasing function in $\big(0, (\frac{p^*_s(p^*_s-p)}{p\beta})^\frac{1}{p-\alpha}\big)$ and $g_1$ is an increasing function in $\big((\frac{p^*_s(p^*_s-p)}{p\beta})^\frac{1}{p-\alpha}, \infty\big)$. Thus, $g_1$ has only one root. Hence the claim follows. Next, we observe that $\alpha<p\Longrightarrow \beta>p^*_s-p$ and therefore, $\lim_{\tau\to\infty}g(\tau)=\infty.$
		\begin{figure}[h]
			\subfigure[The case $3(ii)$]{
				\begin{tikzpicture}[scale=.8]
				\draw[->][] (-.25,0)--(4,0) node[anchor=west] {\tiny{$\tau$}};
				\draw[->][] (0,-2)--(0,2.5) node[anchor=east] {\tiny{$y$}};

				\draw[][-] [domain=0:1.7] plot ({\x}, {1.05*\x^3.5-1.56*\x^1.5-1.21});
				\draw[](2.8,1) node[anchor=south east]{\tiny{$g_1(\tau)$}};
				\draw[](.8,-1.82)--(.8,0) node[anchor=south]{\tiny{$\tau_1$}};
				\end{tikzpicture}
				
				\begin{tikzpicture}[scale=.7]
				\draw[->][] (-.25,0)--(4,0) node[anchor=west] {\tiny{$\tau$}};
				\draw[->][] (0,-.25)--(0,2) node[anchor=east] {\tiny{$y$}};
				
				\draw[][-] [domain=0:3.5] plot ({\x}, {1.1+1.1*\x^2-3.1*\x^1.1});
				\draw[](3.35,2) node[anchor=north east ]{\tiny{$g(\tau)$}};
				\draw[](1.1,0) node[anchor=south east ]{\tiny{$\tau_1$}};
				\draw[](2.9,0) node[anchor=south east ]{\tiny{$\tau_2$}};
				\end{tikzpicture}
			}
		\end{figure}
		Combining all the above observations and claim, it follows that $g$ has only one critical point in $(0,\infty)$ and two roots $\tau_1$, $\tau_2$ with $\tau_1\in(0,1)$ and $\tau_2\in(1,\infty)$. Hence, $h$ has exactly two critical points $\tau_1, \tau_2$. Since $h$ is increasing near $0$  leads to the conclusion that first positive critical point of $h$, i.e., $\tau_1$ is the local maximum for $h$ and $h(\tau_1)>1$ and since, $\lim\limits_{\tau\to\infty}h(\tau)=1$, at the second critical point of $h$ i.e., at $\tau_2$ we have  $h(\tau_2)<1.$ Therefore, in this case $\tau_{min}=\tau_2>0$ is the minimum point of $h$ with $h(\tau_{min})<1.$
		\begin{figure}
			\subfigure[The case $3(ii)$]{
				\begin{tikzpicture}[scale=.85]
				\draw[->][] (-.25,0)--(7,0) node[anchor=west] {\tiny{$\tau$}};
				\draw[->][] (0,-.25)--(0,2) node[anchor=east] {\tiny{$y$}};
				\draw[](7,1)--(0,1) node[anchor=east] {\tiny{$h(0)=1$}};
				
				\draw[][-] [domain=0:6,smooth] plot ({\x}, {8*((1+\x^2.5)/(1+\x^2.5+\x^3.9)^.64)-7});
				\draw[](5,.68) node[anchor=north east ]{\tiny{$h(\tau)$}};
				\draw[](.55,1.2)--(.55,0) node[anchor=north]{\tiny{$\tau_1$}};
				\draw[](2,.47)--(2,0) node[anchor=north]{\tiny{$\tau_2$}};
				\end{tikzpicture}
			}
		\end{figure}
		
		\medskip
		
		(iii) $\beta>p$, $\alpha=p$.
		
		In this case, $g(0)>0$ and $\alpha=p\Longrightarrow \beta=p^*_s-p$. Therefore,
		\begin{align*}
		g^{\prime}(\tau)&=\tau^{\beta-p-1}\big[\alpha\beta\tau^{p}-p^*_s(p^*_s-p)\tau^{\alpha}-\beta(\beta-p)\big]\\
		&=\tau^{\beta-p-1}\big[(\alpha-p^*_s)\beta\tau^\alpha-\beta(\beta-p)\big]<0,
		\end{align*}
		i.e., $g$ is a strictly decreasing function. Also, observe that $\lim_{\tau\to\infty}g(\tau)=-\infty$. Hence, $g$ has only one root in $(0,\infty)$, i.e., $h$ has only critical point $\tau$ in $(0,\infty)$. Since $\beta>p$ implies $h$ is increasing near $0$, so at $\tau$, $h$ attains the maximum with $h(\tau)>1$. Combining this with $\lim_{\tau\to\infty}h(\tau)=1$ proves that $h(\tau_{min})=h(0)=1$, i.e, $\tau_{min}=0$.
		\begin{figure}[h]
			\subfigure[The case $3(iii)$]{
				\begin{tikzpicture}[scale=1]
				\draw[->][] (-.25,0)--(4,0) node[anchor=west] {\tiny{$\tau$}};
				\draw[->][] (0,-.25)--(0,2.5) node[anchor=east] {\tiny{$y$}};
				
				\draw[][-] [domain=0:1.2] plot ({\x}, {1.5-\x^1.5-\x^.9});
				\draw[](1.2,.51) node[anchor=south east]{\tiny{$g(\tau)$}};
				\end{tikzpicture}
				
				\begin{tikzpicture}[scale=.8]
				\draw[->][] (-.25,0)--(7,0) node[anchor=west] {\tiny{$\tau$}};
				\draw[->][] (0,-.25)--(0,2) node[anchor=east] {\tiny{$y$}};
				\draw[](6,1)--(0,1) node[anchor=east] {\tiny{$y=1$}};

				\draw[][-] [domain=0:5,smooth] plot ({\x}, {10*((1+\x^1.5)/(1+\x^1.5+\x^3)^.5)-9});
				\draw[](4,2.15) node[anchor=north east ]{\tiny{$h(\tau)$}};
				\draw[](.99,2.53)--(.99,0) node[anchor=north]{\tiny{$\tau_1$}};
				\end{tikzpicture}
			}
		\end{figure}
		
	\end{proof}

	In order to prove Theorem~\ref{main-gr1}, next we introduce an auxiliary system of equations with a positive parameter $\eta$,
	
	\begin{equation}
	\tag{$\mathcal S_\eta$}\label{system with eta}
	\begin{cases}
	(-\Delta_p)^s u =\eta|u|^{p^*_s-2}u+ \frac{\alpha}{p_s^*}|u|^{\alpha-2}u|v|^{\beta}\;\;\text{in}\;\rn,\\
	(-\Delta_p)^s v =|v|^{p^*_s-2}v+ \frac{\beta}{p_s^*}|v|^{\beta-2}v|u|^{\alpha}\;\;\text{in}\;\rn,
	\\
	u,\;v\in\dot{W}^{s,p}(\rn).
	\end{cases}
	\end{equation}
	We define the following minimization problem associated to \eqref{system with eta}:
	
	\begin{equation*}\label{best constant gamma}
	S_{\eta,\alpha,\beta}:=\inf_{\substack{(u,v)\in X,\\
			(u,v)\neq0}}\frac{\|u\|_{\dot{W}^{s,p}}^p+\|v\|_{\dot{W}^{s,p}}^p}{\left(\displaystyle\int_{\rn}\left(\eta|u|^{p^*_s}+|v|^{p^*_s}+|u|^{\alpha}|v|^{\beta}\right)dx\right)^{\frac{p}{p^*_s}}}.
	\end{equation*}
	Similarly for $\tau>0$, we define $$f_\eta(\tau):=\frac{1+\tau^p}{\left(\eta+\tau^\beta+\tau^{p^*_s}\right)^{p/{p^*_s}}},\quad  f_\eta(\tau_{min}^*)=\min_{\substack{\tau\geq 0}} f_\eta(\tau).$$
	
	Proceeding as in the proof of Lemma~\ref{l:g function}, we find $\epsilon\in(0,1)$ small such that $\tau_{min}^*(\eta), \lambda^*(\eta), \mu^*(\eta)$ are unique for $\eta\in(1-\epsilon, 1+\epsilon)$ and $\tau_{min}^*(\eta)$ satisfies
	$$\tau^{p-1}\left(\eta p^*_s+\alpha\tau^{\beta}-\beta\tau^{\beta-p}-p^*_s\tau^{p^*_s-p}\right)=0.$$
	Moreover, $\tau_{min}^*(\eta), \lambda^*(\eta), \mu^*(\eta)$ are $C^1$ for $\eta\in(1-\epsilon, 1+\epsilon)$ and $\epsilon>0$ small.
	Indeed, if we denote
	$$F(\eta,\tau)=\eta p^*_s+\alpha\tau^{\beta}-\beta\tau^{\beta-p}-p^*_s\tau^{p^*_s-p}.$$
	Then, $$\frac{\partial F}{\partial \eta}=\tau^{\beta-p-1}\big[\alpha\beta\tau^p-p^*_s(p^*_s-p)\tau^\alpha-\beta(\beta-p)\big].$$
	
	Since $\tau_{min}$ is the minimum of $h$, direct computation yields $g(\tau_{min})=0$, $g'(\tau_{min})>0$. Therefore, $F(1,\tau_{min})=0$, $\frac{\partial F}{\partial \eta}(1,\tau_{min})>0$. Consequently, by implicit function theorem, we obtain that $\tau_{min}^*(\eta), \lambda^*(\eta), \mu^*(\eta)$ are $C^1$ for $\eta\in(1-\epsilon, 1+\epsilon)$.
	
	\medskip
	
	\medskip
	
	\noindent{\bf Proof of Theorem~\ref{main-gr1}}: Let $(u_0,v_0)$ is a ground state solution of \eqref{MAT1}. First, we claim that
	\begin{equation}\label{1:Lp* norm equal}
	\int_{\rn}|u_0|^{p^*_s}dx=\lambda^{p^*_s}\int_{\rn}|U|^{p^*_s}dx.
	\end{equation}
	In order to prove this, we define the following min-max problem associated to \eqref{system with eta}
	\begin{equation*}
	B(\eta):=\inf_{\substack{(u,v)\in X\setminus\{0\}}}\max_{\substack{t>0}}E_{\eta}(tu,tv),
	\end{equation*}
	where
	$$
	E_{\eta}(u,v):=\frac{1}{p}\left(\|u\|_{\dot{W}^{s,p}}^p+\|v\|_{\dot{W}^{s,p}}^p\right)-\frac{1}{p^*_s}\int_{\rn}\left(\eta\left(u^+\right)^{p^*_s}+\left(v^+\right)^{p^*_s}+\left(u^+\right)^{\alpha}\left(v^+\right)^{\beta}\right)dx.
	$$	
	Observe that there exists $t(\eta)>0$ such that $\max_{\substack{t>0}}E_{\eta}(tu_0,tv_0)=E_{\eta}(t(\eta)u_0,t(\eta)v_0)$ and moreover, $t(\eta)$ satisfies $H(\eta,t(\eta))=0$, where $H(\eta,t)=t^{p^*_s-p}\left(\eta G+D\right)-C$ with
	$$
	C:=\|u_0\|_{\dot{W}^{s,p}}^p+\|v_0\|_{\dot{W}^{s,p}}^p,\quad D:=\int_{\rn}\left(|v_0|^{p^*_s}+|u_0|^{\alpha}|v_0|^{\beta}\right)dx\quad \text{ and }\,G:=\int_{\rn}|u_0|^{p^*_s}dx.
	$$
	As $(u_0,v_0)$ is a least energy solution of \eqref{MAT1}, then
	$$
	H(1,1)=0,\quad\frac{ \partial H}{\partial t}(1,1)>0\quad \text{ and }\,H(\eta,t(\eta))=0.
	$$
	Thus, by the implicit function theorem, there exists $\epsilon>0$ such that $t(\eta):(1-\epsilon,1+\epsilon)\to\re$ is $C^1$ and
	$$
	t^{\prime}(\eta)=-\frac{\frac{ \partial H}{\partial\eta}}{\frac{ \partial H}{\partial t}}\Bigg|_{\eta=1=t}=-\frac{ G}{(p^*_s-p)(G+D)}.
	$$
	By Taylor expansion, we also have $t(\eta)=1+t^{\prime}(1)(\eta-1)+O\left(|\eta-1|^2\right)$ and thus
	$$
	t^p(\eta)=1+pt^{\prime}(1)(\eta-1)+O\left(|\eta-1|^2\right).
	$$
	Since, $H(1,1)=0$ implies $C=G+D$, and $H(\eta,t(\eta))=0$ implies $C=t(\eta)^{p^*_s-p}\left(\eta G+D\right)$.
	Therefore by definition of $B(\eta)$ and above we obtain
	\begin{multline}\label{upper bdd of B(gamma)}
	B(\eta)\leq E_{\eta}(t(\eta)u_0,t(\eta)v_0)=\frac{t(\eta)^{p}}{p}C-\frac{t(\eta)^{p^*_s}}{p^*_s}\left(\eta G+D\right)=t(\eta)^p\frac{s}{N}C=t(\eta)^p B(1)
	\\
	=B(1)-\frac{p\;G B(1)}{(p^*_s-p)(G+D)}(\eta-1)+O\left(|\eta-1|^2\right).
	\end{multline}
	Now, let us compute $B(1)$ from the definition
	
	\begin{align*}
	B(1)&=\inf_{\substack{(u,v)\in X}}E_1(t_{\text{max}}u,t_{\text{max}}v), \quad\text{ where }t_{\text{max}}^{p^*_s-p}=\frac{\|u\|_{\dot{W}^{s,p}}^p+\|v\|_{\dot{W}^{s,p}}^p}{\displaystyle\int_{\rn}\left(|u|^{p^*_s}+|v|^{p^*_s}+|u|^{\alpha}|v|^{\beta}\right)dx}
	\\
	&=\frac{s}{N} \inf_{\substack{(u,v)\in X}}\left(\frac{\|u\|_{\dot{W}^{s,p}}^p+\|v\|_{\dot{W}^{s,p}}^p}{\left(\displaystyle\int_{\rn}\left(|u|^{p^*_s}+|v|^{p^*_s}+|u|^{\alpha}|v|^{\beta}\right)dx\right)^{\frac{p}{p^*_s}}}\right)^{\frac{p^*_s}{p^*_s-p}}\\
	&=\frac{s}{N} S_{\alpha,\beta}^{\frac{p^*_s}{p^*_s-p}}
	\\
	&=\frac{s}{N}\left(\frac{\|u_0\|_{\dot{W}^{s,p}}^p+\|v_0\|_{\dot{W}^{s,p}}^p}{\left(\displaystyle\int_{\rn}\left(|u_0|^{p^*_s}+|v_0|^{p^*_s}+|u_0|^{\alpha}|v_0|^{\beta}\right)dx\right)^{\frac{p}{p^*_s}}}\right)^{\frac{p^*_s}{p^*_s-p}}\\
	&=\frac{s}{N}(G+D).
	\end{align*}
	Using this in \eqref{upper bdd of B(gamma)}, we obtain
	\begin{equation*}
	B(\eta)\leq B(1)-\frac{G}{p^*_s}(\eta-1)+O\left(|\eta-1|^2\right).
	\end{equation*}	
	Therefore, we have
	\begin{equation*}
	\frac{B(\eta)-B(1)}{\eta-1}
	\begin{cases}
	\leq-\frac{G}{p^*_s}+O\left(|\eta-1|\right) \text{ if }\eta>1,\\
	\geq-\frac{G}{p^*_s}+O\left(|\eta-1|\right) \text{ if }\eta<1.
	\end{cases}
	\end{equation*}	
	This implies that
	\begin{equation}\label{1:derivative of B(1)}
	B^{\prime}(1)=-\frac{G}{p^*_s}=-\frac{1}{p^*_s}\int_{\rn}|u_0|^{p^*_s}dx.
	\end{equation}
	Arguing similarly as in the proof of Lemma~\ref{basic lemma}, it follows that $S_{\eta,\alpha,\beta}$ is attained by $(tU,\tau(\eta)tU)$. Therefore,
	\begin{eqnarray*}
		B(\eta)&=\frac{s}{N}\left(\frac{1+\tau(\eta)^p}{\left(\eta+\tau(\eta)^\beta+\tau(\eta)^{p^*_s}\right)^{\frac{p}{p^*_s}}}\right)^{\frac{p^*_s}{p^*_s-p}}\displaystyle\int_{\rn}|U|^{p^*_s}dx\\
		&=\frac{s}{N}\frac{\left(1+\tau(\eta)^p\right)^{\frac{ n}{sp}}}{\left(\eta+\tau(\eta)^\beta+\tau(\eta)^{p^*_s}\right)^{\frac{n}{sp}-1}}\displaystyle\int_{\rn}|U|^{p^*_s}dx.
	\end{eqnarray*}
	Then, from a simple computation it follows
	
	\begin{multline*}
	B^{\prime}(\eta)=\frac{\left(1+\tau(\eta)^p\right)^{\frac{ n}{sp}-1}}{p^*_s\left(\eta+\tau(\eta)^\beta+\tau(\gamma)^{p^*_s}\right)^{\frac{n}{sp}}}\\
	\times\left[\tau^{\prime}(\eta)\tau(\eta)^{p-1}\left(\eta p^*_s+\alpha\tau(\eta)^{\beta}-\beta\tau(\eta)^{\beta-p}-p^*_s\tau(\eta)^{p^*_s-p}\right)-1-\tau(\eta)^{p}\right]\int_{\rn}|U|^{p^*_s}dx.
	\end{multline*}
	Note that for $\eta=1$,  $\tau(1)$ satisfies the equation $g(\tau)=0$, where $g(\tau)$ is given by \eqref{def: g function}, thus we obtain $\tau(1)=\tau_{min}$. Consequently,
	
	\begin{equation}\label{2:derivative of B(1)}
	B^{\prime}(1)=-\frac{1}{p^*_s}\left(\frac{1+\tau_{min}^p}{1+\tau_{min}^\beta+\tau_{min}^{p^*_s}}\right)^{\frac{p^*_s}{p^*_s-p}}\int_{\rn}|U|^{p^*_s}dx=-\frac{\lambda^{p^*_s}}{p^*_s}\int_{\rn}|U|^{p^*_s}dx.
	\end{equation}
	Combining \eqref{1:derivative of B(1)} and \eqref{2:derivative of B(1)} we conclude \eqref{1:Lp* norm equal}.
	By a similar argument as in the proof of \eqref{1:Lp* norm equal}, we show that
	
	\begin{equation}\label{2:Lp* norm equal}
	\int_{\rn}|v_0|^{p^*_s}dx=\tau_{min}^{p^*_s}\lambda^{p^*_s}\int_{\rn}|U|^{p^*_s}dx,\quad\int_{\rn}|u_0|^{\alpha}|v_0|^{\beta}dx=\tau_{min}^{\beta}\lambda^{p^*_s}\int_{\rn}|U|^{p^*_s}dx.
	\end{equation}
	Therefore, by \eqref{1:Lp* norm equal} and \eqref{2:Lp* norm equal}, we obtain
	\begin{equation*}\label{3:Lp* norm equal}
	\int_{\rn}|u_0|^{\alpha}|v_0|^{\beta}dx=\tau_{min}^{\beta}\int_{\rn}|u_0|^{p^*_s},\quad\int_{\rn}|u_0|^{\alpha}|v_0|^{\beta}dx=\tau_{min}^{\beta-p^*_s}\int_{\rn}|v_0|^{p^*_s}dx.
	\end{equation*}
	Again, since $(\lambda U, \mu U)$  solves the problem \eqref{MAT1}, we get
	
	\begin{equation}\label{1: algbebric reltn}
	\lambda^{p^*_s-p}+\frac{\alpha}{p^*_s}\mu^{\beta}\lambda^{\alpha-p}=1= \mu^{p^*_s-p}+\frac{\beta}{p^*_s}\mu^{\beta-p}\lambda^{\alpha}.
	\end{equation}
	Now define $(u_1,v_1):=\left(\frac{u_0}{\lambda},\frac{v_0}{\mu}  \right)$. Using \eqref{1:Lp* norm equal}, \eqref{2:Lp* norm equal} and \eqref{1: algbebric reltn} we have
	
	\begin{multline*}
	\|u_1\|_{\dot{W}^{s,p}}^p=\lambda^{-p}\|u_0\|_{\dot{W}^{s,p}}^p=\lambda^{-p}\int_{\rn}\left(|u_0|^{p^*_s}+\frac{\alpha}{p^*_s}|u_0|^{\alpha}|v_0|^{\beta} \right)dx
	\\
	=\lambda^{-p}\left(\lambda^{p^*_s}+\frac{\alpha}{p^*_s}\mu^{\beta}\lambda^{\alpha}\right)\int_{\rn}|U|^{p^*_s}dx=\|U\|_{\dot{W}^{s,p}}^p.
	\end{multline*}
	Similarly, we obtain $\|v_1\|_{\dot{W}^{s,p}}^p=\|U\|_{\dot{W}^{s,p}}^p$. Therefore, we have
	
	\begin{equation}\label{seminorms are equal}
	\|u_1\|_{\dot{W}^{s,p}}^p=\|U\|_{\dot{W}^{s,p}}^p=\|v_1\|_{\dot{W}^{s,p}}^p.
	\end{equation}
	Also, by \eqref{1:Lp* norm equal}
	
	\begin{equation}\label{Lp* norm of u1 U equal}
	\int_{\rn}|u_1|^{p^*_s}dx=\int_{\rn}|U|^{p^*_s}dx,
	\end{equation}
	and by \eqref{2:Lp* norm equal},
	
	\begin{equation}\label{Lp* norm of v1 U equal}
	\int_{\rn}|v_1|^{p^*_s}dx=\int_{\rn}|U|^{p^*_s}dx.
	\end{equation}
	Thus, from \eqref{seminorms are equal} and \eqref{Lp* norm of u1 U equal}, we conclude $u_1$ achieves $\mathcal S$. Further, from \eqref{best constant alpha+beta}, \eqref{seminorms are equal} and \eqref{Lp* norm of v1 U equal} implies $v_1$ also achieves $\mathcal S$ in \eqref{best constant alpha+beta}. This completes the proof.
	\smallskip
	
	\section{Proof of Theorem~\ref{minimizer value}, \ref{gamma near 0}, and \ref{in ball}}
	In this section we study the system \eqref{MAT2} which we introduced in the introduction. For the reader's convenience, we recall \eqref{MAT2} below
	\begin{equation*}
	\tag{$\tilde{\mathcal S_{\gamma}}$}
	\begin{cases}
	(-\Delta_p)^s u =|u|^{p^*_s-2}u+ \frac{\alpha\gamma}{p_s^*}|u|^{\alpha-2}u|v|^{\beta}\;\;\text{in}\;\rn,\vspace{.2cm}\\
	(-\Delta_p)^s v =|v|^{p^*_s-2}v+ \frac{\beta\gamma}{p_s^*}|v|^{\beta-2}v|u|^{\alpha}\;\;\text{in}\;\rn,
	\\
	u,\;v\in\dot{W}^{s,p}(\rn).
	\end{cases}
	\end{equation*}
	We also recall that (see \eqref{J-functional}) the energy functional associated to the above system is
	$$\mathcal J(u,v)=\frac{1}{p}\left(\|u\|_{\dot{W}^{s,p}}^p+\|v\|_{\dot{W}^{s,p}}^p\right)-\frac{1}{p^*_s}\int_{\rn}\left(|u|^{p^*_s}+|v|^{p^*_s}+\gamma|u|^{\alpha}|v|^{\beta}\right)dx, \quad(u,v)\in X.$$
	and the Nehari manifold \eqref{Nehari}
	\begin{multline}
	\mathcal N=\left\{(u,v)\in X:u\neq0,\;v\neq0,\;\|u\|_{\dot{W}^{s,p}}^p=\int_{\rn}\left(|u|^{p^*_s}+\frac{ \alpha\gamma}{p^*_s}|u|^{\alpha}|v|^{\beta}\right)dx\right.,\\
	\left. \|v\|_{\dot{W}^{s,p}}^p=\int_{\rn}\left(|v|^{p^*_s}+\frac{ \beta\gamma}{p^*_s}|u|^{\alpha}|v|^{\beta}\right)dx\right\}.\nonumber
	\end{multline}
	Therefore, it follows
	\begin{multline*}
	A=\inf\limits_{(u,v)\in\mathcal N}\mathcal{J}(u,v)
	=\inf\limits_{(u,v)\in\mathcal N}\frac{s}{N} \left(\|u\|_{\dot{W}^{s,p}}^p+\|v\|_{\dot{W}^{s,p}}^p\right)\\
	=\inf\limits_{(u,v)\in\mathcal N}\frac{s}{N}\int_{\rn}\left(|u|^{p^*_s}+|v|^{p^*_s}+\gamma|u|^{\alpha}|v|^{\beta}\right)dx.
	\end{multline*} 
	
	\begin{proposition}\label{algebraic propoistion for alpha beta bigger than p}
		Assume that $c,\;d\in\re$ satisfy
		\begin{equation}\label{algebraic eq2}
		\begin{cases}
		c^{\frac{p^*_s-p}{p}}+\frac{\alpha\gamma}{p^*_s}c^{\frac{\alpha-p}{p}}d^{\frac{\beta}{p}}\geq1,\\
		d^{\frac{p^*_s-p}{p}}+\frac{\beta\gamma}{p^*_s}d^{\frac{\beta-p}{p}}c^{\frac{\alpha}{p}}\geq1,\\
		c,d>0.
		\end{cases}
		\end{equation}
		If $\frac{N}{2s}<p<\frac{N}{s}$, $\alpha,\beta>p$ and \eqref{condition on gamma: 1} holds then $c+d\geq k+\ell$, where $k,\ell\in\re$ satisfy \eqref{algebraic eq1}.
	\end{proposition}
	\begin{proof}
		We use the change of variables $y=c+d,\;x=c/d,\;y_0=k+\ell,\text{ and }x_0=k/\ell$ into \eqref{algebraic eq2} and \eqref{algebraic eq1}, we obtain
		
		$$
		y^{\frac{p^*_s-p}{p}}\geq\frac{(x+1)^{\frac{ p^*_s-p}{p}}}{x^{\frac{ p^*_s-p}{p}}+\frac{ \alpha\gamma}{p^*_s}x^{\frac{\alpha-p}{p}}}=:f_1(x),\;\; y_0^{\frac{ p^*_s-p}{p}}=f_1(x_0),
		$$
		
		$$
		y^{\frac{p^*_s-p}{p}}\geq\frac{(x+1)^{\frac{ p^*_s-p}{p}}}{1+\frac{ \beta\gamma}{p^*_s}x^{\frac{\alpha}{p}}}=:f_2(x),\;\; y_0^{\frac{ p^*_s-p}{p}}=f_2(x_0).
		$$
		
		Then, one has
		\begin{multline*}
		f_1^{\prime}(x)=\frac{\alpha\gamma(x+1)^{\frac{ p^*_s-2p}{p}}x^\frac{ \alpha-2p}{p}}{pp^*_s\left(x^{\frac{ p^*_s-p}{p}}+\frac{ \alpha\gamma}{p^*_s}x^{\frac{\alpha-p}{p}}\right)^2}\left[-\frac{p^*_s(p^*_s-p)}{\alpha\gamma}x^{\frac{\beta}{p}}+\beta x-\alpha+p\right]\\
		=:\frac{\alpha\gamma(x+1)^{\frac{ p^*_s-2p}{p}}x^\frac{ \alpha-2p}{p}}{pp^*_s\left(x^{\frac{ p^*_s-p}{p}}+\frac{ \alpha\gamma}{p^*_s}x^{\frac{\alpha-p}{p}}\right)^2}\;g_1(x),
		\end{multline*}
		\begin{multline*}
		f_2^{\prime}(x)=\frac{\beta\gamma(x+1)^{\frac{ p^*_s-2p}{p}}}{pp^*_s\left(1+\frac{ \beta\gamma}{p^*_s}x^{\frac{\alpha}{p}}\right)^2}\left[\frac{p^*_s(p^*_s-p)}{\beta\gamma}+(\beta-p)x^{\frac{\alpha}{p}}-\alpha x^{\frac{\alpha-p}{p}}\right]\\
		=:\frac{\beta\gamma(x+1)^{\frac{ p^*_s-2p}{p}}}{pp^*_s\left(1+\frac{ \beta\gamma}{p^*_s}x^{\frac{\alpha}{p}}\right)^2}\;g_2(x).
		\end{multline*}
		Hence, we obtain $x_1=\left(\frac{ p\alpha\gamma}{p^*_s(p^*_s-p)}\right)^{\frac{p}{\beta-p}}$ from $g_1^{\prime}(x)=0$ and similarly, for $g_2$ we have $x_2=\frac{\alpha-p}{\beta-p}$. Now using \eqref{condition on gamma: 1} we conclude that
		
		\begin{equation*}
		\max_{\substack{x>0}}g_1(x)=g_1(x_1)=\left(\frac{ p\alpha\gamma}{p^*_s(p^*_s-p)}\right)^{\frac{p}{\beta-p}}(\beta-p)-(\alpha-p)\leq0,
		\end{equation*}
		
		\begin{equation*}
		\min_{\substack{x>0}}g_2(x)=g_2(x_2)=\frac{p^*_s(p^*_s-p)}{\beta\gamma}-p\left(\frac{\alpha-p}{\beta-p}\right)^{\frac{\alpha-p}{p}}\geq0.
		\end{equation*}
		Therefore, we conclude that the function $f_1$ is decreasing in $(0,\infty)$ and on the other hand the function $f_2$ is increasing in $(0,\infty)$. Thus, we have
		
		\begin{multline*}
		y^{\frac{p^*_s-p}{p}}\geq\max\{f_1(x),\;f_2(x)\}\geq\min_{\substack{x>0}}\left(\max\{f_1(x),\;f_2(x)\}\right)
		\\
		=\min_{\substack{\{f_1=f_2\}}}\left(\max\{f_1(x),\;f_2(x)\}\right)=y_0^{\frac{p^*_s-p}{p}}.
		\end{multline*}
		Hence the result follows.	
	\end{proof}
	
	We define the functions
	
	\begin{equation}\label{algebraic eq3}
	\begin{cases}
	F_1(k,\ell):=k^{\frac{p^*_s-p}{p}}+\frac{\alpha\gamma}{p^*_s}k^{\frac{\alpha-p}{p}}\ell^{\frac{\beta}{p}}-1,\quad k>0,\;\ell\geq0,
	\\
	F_2(k,\ell):=\ell^{\frac{p^*_s-p}{p}}+\frac{\beta\gamma}{p^*_s}\ell^{\frac{\beta-p}{p}}k^{\frac{\alpha}{p}}-1,\quad k\geq0,\;\ell>0,
	\\
	\ell(k):=\left(\frac{p^*_s}{\alpha\gamma}\right)^{\frac{p}{\beta}}k^{\frac{p-\alpha}{\beta}}\left(1-k^{\frac{ p^*_s-p}{p}}\right)^{\frac{p}{\beta}},\quad 0<k\leq1,
	\\
	k(\ell):=\left(\frac{p^*_s}{\beta\gamma}\right)^{\frac{p}{\alpha}}\ell^{\frac{p-\beta}{\alpha}}\left(1-\ell^{\frac{ p^*_s-p}{p}}\right)^{\frac{p}{\alpha}},\quad 0<\ell\leq1.
	\end{cases}
	\end{equation}
	Then $F_1(k,\ell(k))=0$ and $F_2(k(\ell),\ell)=0.$
	
	\begin{lemma}\label{Sol for F1 F2}
		Assume that $\frac{2N}{N+2s}<p<\frac{N}{2s}$ and $\alpha,\;\beta<p$. Then
		\begin{equation}\label{F1=0 F2=0}
		F_1(k,\ell)=0,\;F_2(k,\ell)=0,\quad k,\,  \ell>0,
		\end{equation}
		has a solution $(k_0,\ell_0)$ such that  $F_2(k,\ell(k))<0$ for all $k\in(0,k_0)$, that is $(k_0,\ell_0)$ satisfies \eqref{condition on k_0}. Similarly, \eqref{F1=0 F2=0} has a solution $(k_1,\ell_1)$ such that $F_1(k(\ell),\ell)<0$ for all $\ell\in(0,\ell_1)$ that is $(k_1,\ell_1)$ satisfies \eqref{algebraic eq1} and
		$
		\ell_1=\min\{\ell:(k,\ell)\text{ satisfies }\eqref{algebraic eq1}\}.
		$
	\end{lemma}
	\begin{proof}
		The proof is exactly similar to \cite[Lemma~3.2]{GuPeZo}.
	\end{proof}
	
	\begin{lemma}
		Assume that $\frac{N}{N+2s}<p<\frac{N}{2s}$; $\alpha,\;\beta<p$ and \eqref{condition on gamma: 2} holds. Then $k_0+\ell_0<1$, where $(k_0,\ell_0)$ is same as in Lemma~\ref{Sol for F1 F2} and
		\begin{equation*}
		F_1(k(\ell),\ell)<0\quad\forall\, \ell\in(0,\ell_0),\qquad F_2(k,\ell(k))<0\quad\forall\, k\in(0,k_0).
		\end{equation*}
	\end{lemma}
	
	\begin{proof}
		Using \eqref{algebraic eq3}, we obtain
		
		$$
		\ell^{\prime}(k)=\left(\frac{p^*_s}{\alpha\gamma}\right)^{\frac{p}{\beta}}k^{\frac{p-p^*_s}{\beta}}\left(1-k^{\frac{p^*_s-p}{p}}\right)^{\frac{p-\beta}{\beta}}\left(\frac{p-\alpha}{\beta}-k^{\frac{p^*_s-p}{p}}\right),
		$$
		
		and then we have
		
		\begin{multline*}
		\ell^{\prime\prime}(k)=\frac{(p-\beta)(p^*_s-p)}{p\beta} \left(\frac{p^*_s}{\alpha\gamma}\right)^{\frac{p}{\beta}}k^{\frac{p-2\beta-\alpha}{\beta}}\left(1-k^{\frac{p^*_s-p}{p}}\right)^{\frac{p-2\beta}{\beta}}\left[k^{\frac{p^*_s-p}{p}}\left(-\frac{p-\alpha}{\beta}+k^{\frac{p^*_s-p}{p}}\right)\right.\\
		\left.-\left(1-k^{\frac{p^*_s-p}{p}}\right)\left(\frac{p(p-\alpha)}{\beta(p-\beta)}-k^{\frac{p^*_s-p}{p}}\right)\right].
		\end{multline*}
		Note that $\ell^{\prime}(1)=0=\ell^{\prime}\left(\left(\frac{p-\alpha}{\beta}\right)^{\frac{ p}{p^*_s-p}}\right)$ and $\ell^{\prime}(k)>0$ for $0<k<\left(\frac{p-\alpha}{\beta}\right)^{\frac{ p}{p^*_s-p}}$, whereas $\ell^{\prime}(k)<0$ for $\left(\frac{p-\alpha}{\beta}\right)^{\frac{ p}{p^*_s-p}}<k<1$.
		From $\ell^{\prime\prime}(k)=0$, we obtain $\tilde{k}=\left(\frac{p(p-\alpha)}{\beta(2p-p^*_s)}\right)^{\frac{ p}{p^*_s-p}}$. Then by \eqref{condition on gamma: 2}, we obtain
		\begin{multline*}
		\min_{\substack{k\in(0,1]}}\ell^{\prime}(k)=\min_{\substack{k\in\left(\left(\frac{p-\alpha}{\beta}\right)^{\frac{ p}{p^*_s-p}},\;1\right]}}\ell^{\prime}(k)=\ell^{\prime}(\tilde{k})=-\left(\frac{p^*_s(p^*_s-p)}{\alpha\gamma p}\right)^{\frac{ p}{\beta}}\left(\frac{p-\beta}{p-\alpha}\right)^{\frac{p-\beta}{\beta}}\geq -1.
		\end{multline*}	
		The remaining proof follows from \cite[Lemma~3.3]{GuPeZo} by considering $\mu_1=1=\mu_2$ in their proof.
	\end{proof}
	
	\begin{lemma}\label{algebraic lemma for alpha beta less than p}
		Assume that $\frac{N}{N+2s}<p<\frac{N}{2s}$; $\alpha,\;\beta<p$ and \eqref{condition on gamma: 2} holds. Then
		\begin{equation*}
		\begin{cases}
		k+\ell\leq k_0+\ell_0,\\
		F_1(k,\ell)\geq0,\;\;F_2(k,\ell)\geq0,\\
		k,\;\ell\geq0\;\;(k,\ell)\neq(0,0),
		\end{cases}
		\end{equation*}
		has a unique solution $(k,\ell)=(k_0,\ell_0)$, where $F_1,\;F_2$ are given by \eqref{algebraic eq3}.
	\end{lemma}
	\begin{proof}
		The proof follows from \cite[Proposition~3.4]{GuPeZo}.
	\end{proof}
	
	\noindent{\bf Proof of Theorem~\ref{minimizer value}}: Using \eqref{algebraic eq1}, we have $\left(k_0^{1/p}U,\ell_0^{1/p}U\right)\in\mathcal{N}$ is a nontrivial solution of \eqref{MAT2} and
	\begin{equation}\label{upperbdd of A}
	A\leq\mathcal{J}\left(k_0^{1/p}U,\ell_0^{1/p}U\right)=\frac{s}{N}(k_0+\ell_0)\mathcal{S}^{\frac{N}{sp}}.
	\end{equation}
	Now, suppose $\{(u_n,v_n)\}\in\mathcal{N}$ be a minimizing sequence for $A$ such that $\mathcal{J}(u_n,v_n)\to A$ as $n\to\infty.$ Let $c_n=\|u_n\|_{L^{p^*_s}(\rn)}^p$ and $d_n=\|v_n\|_{L^{p^*_s}(\rn)}^p$. Then by H\"older inequality we have
	\begin{equation}\label{Sc_n upperbdd}
	\mathcal{S}c_n\leq\|u_n\|_{\dot{W}^{s,p}}^p=\int_{\rn}\left(|u_n|^{p^*_s}+\frac{ \alpha\gamma}{p^*_s}|u_n|^{\alpha}|v_n|^{\beta}\right)dx\leq c_n^{\frac{ p^*_s}{p}}+\frac{ \alpha\gamma}{p^*_s}c_n^{\frac{\alpha}{p}}d_n^{\frac{\beta}{p}}.
	\end{equation}
	This implies that
	$$
	\tilde{c_n}^{\frac{ p^*_s-p}{p}}+\frac{ \alpha\gamma}{p^*_s}\tilde{c_n}^{\frac{\alpha-p}{p}}\tilde{d_n}^{\frac{\beta}{p}}\geq1\quad \text{i.e., }\, F_1(\tilde{c_n},\tilde{d_n})\geq0,
	$$
	where $\tilde{c_n}=\frac{c_n}{\mathcal{S}^{\frac{p}{p^*_s-p}}}$, $\tilde{d_n}=\frac{d_n}{\mathcal{S}^{\frac{p}{p^*_s-p}}}.$
	Similarly, we get
	
	\begin{equation}\label{Sd_n upperbdd}
	\mathcal{S}d_n\leq\|v_n\|_{\dot{W}^{s,p}}^p=\int_{\rn}\left(|v_n|^{p^*_s}+\frac{ \beta\gamma}{p^*_s}|u_n|^{\alpha}|v_n|^{\beta}\right)dx\leq d_n^{\frac{ p^*_s}{p}}+\frac{ \beta\gamma}{p^*_s}c_n^{\frac{\alpha}{p}}d_n^{\frac{\beta}{p}},
	\end{equation}
	and thus $F_2(\tilde{c_n},\tilde{d_n})\geq0.$
	Then for $\alpha,\beta>p$, by Proposition~\ref{algebraic propoistion for alpha beta bigger than p} we have $\tilde{c_n}+\tilde{d_n}\geq k+\ell=k_0+\ell_0$, on the other hand for $\alpha,\beta<p$, by Lemma~\ref{algebraic lemma for alpha beta less than p} we have $\tilde{c_n}+\tilde{d_n}= k_0+\ell_0.$ Hence,
	
	\begin{equation}\label{lowerbdd of c_n+d_n}
	c_n+d_n\geq(k_0+\ell_0)\mathcal{S}^{\frac{N-sp}{sp}}.
	\end{equation}
	Since $\mathcal{J}(u_n,v_n)=\frac{s}{N}\left(\|u_n\|_{\dot{W}^{s,p}}^p+\|v\|_{\dot{W}^{s,p}}^p\right)$, using \eqref{upperbdd of A}-\eqref{Sd_n upperbdd} we have
	
	$$
	\mathcal{S}(c_n+d_n)\leq\frac{N}{s}\mathcal{J}(u_n,v_n)=\frac{N}{s}A+o(1)\leq(k_0+\ell_0)\mathcal{S}^{\frac{N}{sp}}+o(1)
	$$
	This implies that
	
	\begin{equation}\label{upperbdd of c_n+d_n}
	c_n+d_n\leq(k_0+\ell_0)\mathcal{S}^{\frac{N-sp}{sp}}+o(1).
	\end{equation}
	Combining \eqref{lowerbdd of c_n+d_n} and \eqref{upperbdd of c_n+d_n}, we obtain $c_n+d_n\to(k_0+\ell_0)\mathcal{S}^{\frac{N-sp}{sp}}$ as $n\to\infty.$ Therefore,
	
	$$
	A=\lim\limits_{n\to\infty}\mathcal{J}(u_n,v_n)\geq\frac{s}{N}\mathcal{S}\lim\limits_{n\to\infty}(c_n+d_n)=(k_0+\ell_0)\mathcal{S}^{\frac{N-sp}{p}}.
	$$
	Therefore,
	$$
	A=\frac{s}{N}(k_0+\ell_0)\mathcal{S}^{\frac{N}{sp}}=\mathcal{J}\left(k_0^{1/p}U,\ell_0^{1/p}U\right).
	$$
	This completes the proof of Theorem~\ref{minimizer value}.
	\hfill{$\square$}

	\medskip
	
	\medskip

	Next, we prove existence of solutions of \eqref{system_ball}, namely Theorem~\ref{in ball}. For this, 
	define $$X(B_R(0))=W^{s,p}_{0}(B_R(0))\times W^{s,p}_{0}(B_R(0)),$$ 
	where $W^{s,p}_{0}(B_R(0))=\{u\in W^{s,p}(\rn): u=0 \mbox{ in } \rn\setminus B_R(0)\}$ with the norm $||\cdot||_{\dot{W}^{s,p}}$,
	and
	\begin{multline*}
	\tilde{\mathcal N}(R)=\left\{(u,v)\in X(B_R(0))\setminus\{(0,0)\}:\, \|u\|_{\dot{W}^{s,p}}^p+\|v\|_{\dot{W}^{s,p}}^p\right.
	\\
	\left.=\int_{B_R(0)}\left(|u|^{p^*_s}+|v|^{p^*_s}+\gamma|u|^{\alpha}|v|^{\beta}\right)dx\right\},
	\end{multline*}
	and set $\tilde{A}(R):=\inf\limits_{(u,v)\in\tilde{\mathcal N}(R)}\mathcal{J}(u,v)$. We also define
	
	\begin{equation*}
	\tilde{\mathcal N}=\left\{(u,v)\in X\setminus\{0\}:\, \|u\|_{\dot{W}^{s,p}}^p+\|v\|_{\dot{W}^{s,p}}^p=\int_{\rn}\left(|u|^{p^*_s}+|v|^{p^*_s}+\gamma|u|^{\alpha}|v|^{\beta}\right)dx\right\}.
	\end{equation*}
	Set $\tilde{A}:=\inf\limits_{(u,v)\in\tilde{\mathcal N}}\mathcal{J}(u,v)$. Since $\mathcal{N}\subset\tilde{\mathcal N}$, it follows $\tilde{A}\leq A$ and by the fractional Sobolev embedding $\tilde{A}>0.$
	
	For $\epsilon\in (0, \min\{\alpha, \beta\}-1)$, consider

	\begin{equation}\label{epsilon system}
	\begin{cases}
	(-\Delta_p)^s u =|u|^{p^*_s-2-2\epsilon}u+ \frac{(\alpha-\epsilon)\gamma}{p_s^*-2\epsilon}|u|^{\alpha-2-\epsilon}u|v|^{\beta-\epsilon}\;\;\text{in}\;B_R(0),\vspace{.2cm}
	\\
	(-\Delta_p)^s v =|v|^{p^*_s-2-2\epsilon}v+ \frac{(\beta-\epsilon)\gamma}{p_s^*-2\epsilon}|v|^{\beta-2-\epsilon}v|u|^{\alpha-\epsilon}\;\;\text{in}\;B_R(0),
	\\
	u,\;v\in W^{s,p}_{0}(B_R(0)).
	\end{cases}
	\end{equation}
	The corresponding energy functional of the system \eqref{epsilon system} is given by
	
	\begin{multline*}
	\mathcal J_\epsilon(u,v):=\frac{1}{p}\left(\|u\|_{\dot{W}^{s,p}}^p+\|v\|_{\dot{W}^{s,p}}^p\right)
	\\
	-\frac{1}{p^*_s-2\epsilon}\int_{B_R(0)}\left(|u|^{p^*_s-2\epsilon}+|v|^{p^*_s-2\epsilon}+\gamma|u|^{\alpha-\epsilon}|v|^{\beta-\epsilon}\right)dx.
	\end{multline*}
	Define
	\begin{multline*}
	\tilde{\mathcal N}_\epsilon(R):=\left\{(u,v)\in X(B_1(0))\setminus\{(0,0)\}:G_\epsilon(u,v):=\|u\|_{\dot{W}^{s,p}}^p+\|v\|_{\dot{W}^{s,p}}^p\right.
	\\
	\left.-\int_{B_R(0)}\left(|u|^{p^*_s-2\epsilon}+|v|^{p^*_s-2\epsilon}+\gamma|u|^{\alpha-\epsilon}|v|^{\beta-\epsilon}\right)dx=0\right\},
	\end{multline*}
	and set $\tilde{A}_\epsilon(R):=\inf\limits_{(u,v)\in\tilde{\mathcal N}_\epsilon(R)}\mathcal{J}_\epsilon(u,v)$.
	
	\begin{lemma} \label{Lemma 3.5}
		For any $\epsilon_0 \in (0,\min\{\alpha - 1,\beta - 1,(p_s^* - p)/2\})$, there exists a constant $C_{\epsilon_0} > 0$ such that
		\[
		\tilde{A}_\epsilon(R) \ge C_{\epsilon_0} \quad \forall \epsilon \in (0,\epsilon_0].
		\]
	\end{lemma}
	
	\begin{proof}
		Let $(u,v)\in\tilde{\mathcal N}_\epsilon(R)$. Then
		\[
		\mathcal J_\epsilon(u,v) = \left(\frac{1}{p} - \frac{1}{p_s^* - 2 \epsilon}\right) \left(\|u\|_{\dot{W}^{s,p}}^p +\|v\|_{\dot{W}^{s,p}}^p\right),
		\]
		so it suffices to show that $\|u\|_{\dot{W}^{s,p}}^p + \|v\|_{\dot{W}^{s,p}}^p$ is bounded away from zero. We have
		\begin{multline} \label{3.13}
		\|u\|_{\dot{W}^{s,p}}^p +\|v\|_{\dot{W}^{s,p}}^p = \int_{B_R(0)} \left(|u|^{p_s^* - 2 \epsilon} + |v|^{p_s^* - 2 \epsilon} + \gamma |u|^{\alpha - \epsilon} |v|^{\beta - \epsilon}\right) dx\\
		\le |B_R(0)|^{2 \epsilon/p_s^*} \left[\left(\int_{B_R(0)} |u|^{p_s^*}\, dx\right)^{(p_s^* - 2 \epsilon)/p_s^*} + \left(\int_{B_R(0)} |v|^{p_s^*}\, dx\right)^{(p_s^* - 2 \epsilon)/p_s^*}\right.\\
		\left.+ \gamma \left(\int_{B_R(0)} |u|^{p_s^*}\, dx\right)^{(\alpha - \epsilon)/p_s^*} \left(\int_{B_R(0)} |u|^{p_s^*}\, dx\right)^{(\beta - \epsilon)/p_s^*}\right]\\
		\le |B_R(0)|^{2 \epsilon/p_s^*}\, \mathcal{S}^{-(p_s^* - 2 \epsilon)/p} \left(\|u\|_{\dot{W}^{s,p}}^{p_s^* - 2 \epsilon} + \|v\|_{\dot{W}^{s,p}}^{p_s^* - 2 \epsilon} + \gamma\, \|u\|_{\dot{W}^{s,p}}^{\alpha - \epsilon}\, \|v\|_{\dot{W}^{s,p}}^{\beta - \epsilon}\right)
		\end{multline}
		by the H\"{o}lder and Sobolev inequalities. By Young's inequality,
		\[
		\|u\|_{\dot{W}^{s,p}}^{\alpha - \epsilon}\,  \|v\|_{\dot{W}^{s,p}}^{\beta - \epsilon} \le \frac{\alpha - \epsilon}{p_s^* - 2 \epsilon}\, \|u\|_{\dot{W}^{s,p}}^{p_s^* - 2 \epsilon} + \frac{\beta - \epsilon}{p_s^* - 2 \epsilon}\,  \|v\|_{\dot{W}^{s,p}}^{p_s^* - 2 \epsilon} \le  \|u\|_{\dot{W}^{s,p}}^{p_s^* - 2 \epsilon} +  \|v\|_{\dot{W}^{s,p}}^{p_s^* - 2 \epsilon}.
		\]
		Therefore, \eqref{3.13} gives
		\begin{equation} \label{3.14}
		\|u\|_{\dot{W}^{s,p}}^p +\|v\|_{\dot{W}^{s,p}}^p \le (1 + \gamma)\, |B_R(0)|^{2 \epsilon/p_s^*}\, \mathcal{S}^{-(p_s^* - 2 \epsilon)/p} \left( \|u\|_{\dot{W}^{s,p}}^{p_s^* - 2 \epsilon} + \|v\|_{\dot{W}^{s,p}}^{p_s^* - 2 \epsilon}\right).
		\end{equation}
		Since $(p_s^* - 2 \epsilon)/p > 1$,
		\[
		\|u\|_{\dot{W}^{s,p}}^{p_s^* - 2 \epsilon} + \|v\|_{\dot{W}^{s,p}}^{p_s^* - 2 \epsilon} \le \left(\|u\|_{\dot{W}^{s,p}}^p +\|v\|_{\dot{W}^{s,p}}^p\right)^{(p_s^* - 2 \epsilon)/p},
		\]
		thus \eqref{3.14} gives
		\[
		\|u\|_{\dot{W}^{s,p}}^p + \|v\|_{\dot{W}^{s,p}}^p \ge \left(\frac{\mathcal{S}^{(p_s^* - 2 \epsilon)/p}}{(1 + \gamma)\, |B_R(0)|^{2 \epsilon/p_s^*}}\right)^{p/(p_s^* - p - 2 \epsilon)}.
		\]
		The desired conclusion follows from this since $p_s^* - p - 2 \epsilon \ge p_s^* - p - 2 \epsilon_0 > 0$ and the function $h(t)=\left(\frac{\mathcal{S}^{(p_s^* - 2t)/p}}{(1 + \gamma)\, |B_R(0)|^{2t/p_s^*}}\right)^{p/(p_s^* - p - 2t)}$ is contiuous and positive in $[0,\epsilon_0]$.
	\end{proof}
	
	\begin{lemma}\label{upper bdd of tilde A epsilon}
		Assume that $\frac{2N}{N+2s}<p<\frac{N}{2s}$ and $\alpha,\beta<p.$ For $\epsilon\in (0, \min\{\alpha, \beta\}-1)$, it holds
		$$
		\tilde{A}_\epsilon(R)<\min\left\{\inf\limits_{(u,0)\in\tilde{\mathcal N}_\epsilon(R)}\mathcal{J}_\epsilon(u,0),\;\inf\limits_{(0,v)\in\tilde{\mathcal N}_\epsilon(R)}\mathcal{J}_\epsilon(0,v)\right\}.
		$$
	\end{lemma}
	\begin{proof}
		Clearly, $2<p^*_s-2\epsilon<p^*_s$ from $\min\{\alpha, \beta\}\leq\frac{p^*_s}{2}.$ Then we may asuume that $u_1$ is a least energy solution of
		\begin{equation*}
		\begin{cases}
		(-\Delta_p)^su=|u|^{p^*_s-2-2\epsilon}u\,\text{ in  }B_R(0),\\
		u\in W^{s,p}_0(B_R(0)).
		\end{cases}
		\end{equation*}
		Set
		$$
		\mathcal{J}_{\epsilon}(u_1,0)=a_{10}:=\inf\limits_{(u,0)\in\tilde{\mathcal N}_\epsilon(R)}\mathcal{J}_\epsilon(u,0),\,\,\,\mathcal{J}_\epsilon(0,u_1)=a_{01}:=\inf\limits_{(0,v)\in\tilde{\mathcal N}_\epsilon(R)}\mathcal{J}_\epsilon(0,v).
		$$
		We claim that for any $\sigma\in\re$, there exists a unique $t(\sigma)>0$ such that $\left(t(\sigma)^{1/p}u_1, t(\sigma)^{1/p}\sigma u_1\right)\in\tilde{\mathcal N}_\epsilon(R)$.
		\begin{align*}
		t(\sigma)^{\frac{ p^*_s-p-2\epsilon}{p}}
		&=\frac{\|u_1\|_{\dot{W}^{s,p}}^p+|\sigma|^p\|u_1\|_{\dot{W}^{s,p}}^p}{\displaystyle\int_{B_R(0)}\left(|u_1|^{p^*_s-2\epsilon}+|\sigma u_1|^{p^*_s-2\epsilon}+\gamma|u_1|^{\alpha-\epsilon}|\sigma  u_1|^{\beta-\epsilon}\right)dx}
		\\
		&=\frac{qa_{10}+qa_{01}|\sigma|^p}{qa_{10}+qa_{01}|\sigma|^{p^*_s-2\epsilon}+|\sigma|^{\beta-\epsilon}\gamma\displaystyle\int_{B_R(0)}|u_1|^{p^*_s-\epsilon}}dx,
		\end{align*}
		where $q:=\frac{ p(p^*_s-2\epsilon)}{p^*_s-p-2\epsilon}$, i.e., $\frac{1}{q}=\frac{1}{p}-\frac{1}{p^*_s-2\epsilon} $. Note that $t(0)=1$, we have
		$$
		\lim\limits_{\sigma\to0}\frac{ t^{\prime}(\sigma)}{|\sigma|^{\beta-2-\epsilon}\sigma}
		=-\frac{ (\beta-\epsilon)\gamma\int_{B_R(0)}|u_1|^{p^*_s-\epsilon}dx}{a_{10}(p^*_s-2\epsilon)}.
		$$
		This implies that as $\sigma\to0$
		$$
		t^{\prime}(\sigma)
		=-\frac{ (\beta-\epsilon)\gamma\displaystyle\int_{B_R(0)}|u_1|^{p^*_s-\epsilon}dx}{a_{10}(p^*_s-2\epsilon)}|\sigma|^{\beta-2-\epsilon}\sigma(1+o(1)).
		$$
		Then
		$$
		t(\sigma)
		=1-\frac{ \gamma\displaystyle\int_{B_R(0)}|u_1|^{p^*_s-\epsilon}dx}{a_{10}(p^*_s-2\epsilon)}|\sigma|^{\beta-\epsilon}(1+o(1))\,\text{ as }\sigma\to0,
		$$
		and therefore
		$$
		t(\sigma)^{\frac{ p^*_s-2\epsilon}{p}}
		=1-\frac{ \gamma\displaystyle\int_{B_R(0)}|u_1|^{p^*_s-\epsilon}dx}{pa_{10}}|\sigma|^{\beta-\epsilon}(1+o(1))\,\text{ as }\sigma\to0.
		$$
		We get for $|\sigma|$ small enough
		\begin{multline*}
		\tilde{A}_\epsilon(R)
		\leq\mathcal{J}_\epsilon\left(t(\sigma)^{1/p}u_1, t(\sigma)^{1/p}\sigma u_1\right)\\
		=\frac{1}{q}\left(qa_{10}+qa_{01}|\sigma|^{p^*_s-2\epsilon}+|\sigma|^{\beta-\epsilon}\gamma\int_{B_R(0)}|u_1|^{p^*_s-\epsilon}dx\right)t(\sigma)^{\frac{ p^*_s-2\epsilon}{p}}\\
		=a_{10}-\frac{1}{p^*_s-2\epsilon}|\sigma|^{\beta-\epsilon}\gamma\int_{B_R(0)}|u_1|^{p^*_s-\epsilon}dx+o(|\sigma|^{\beta-\epsilon})
		<a_{10}.
		\end{multline*}
		Simililarly, we see that $\tilde{A}_\epsilon(R)<a_{01}$. This completes the proof.
	\end{proof}
	
	Note that, similarly to Lemma \ref{upper bdd of tilde A epsilon}, we obtain
	\begin{equation}\label{tl-A-U}
	\tilde A<\min\{\inf_{(u,0)\in\tilde{\mathcal{ N}}} \mathcal{J}(u,0),\, \inf_{(0,v)\in\mathcal{\tilde N}} \mathcal{J}(0,v)  \}=\min\{ \mathcal{J}(U,0),\, \mathcal{J}(0,U)  \}=\frac{s}{N}\,\mathcal{S}^\frac{N}{sp}.
	\end{equation}
	
	\begin{proposition}\label{radially symmetric solution}
		For $0<\epsilon<\min\big\{\min\{\alpha,\beta\}-1, \frac{ p^*_s-p}{2}\big\}$, the system \eqref{epsilon system} has a positive least energy solution $(u_\epsilon,v_\epsilon)$ with $u_\epsilon,\;v_\epsilon$ are both radially symmetric nonincreasing functions.
	\end{proposition}
	\begin{proof}
		By Lemma \ref{Lemma 3.5}, $\tilde{A}_\epsilon(R)>0$. Let $(u,v)\in\tilde{\mathcal N}_\epsilon(R)$ with $u,\, v\geq0.$ We denote $u^*,\,v^*$ be Schwartz symmetrization of $u,\,v$ respectively. Then by nonlocal P\'{o}lya-Szeg\"{o} inequality (see, \cite{AlLi}) and properties of the Schwartz symmetrization, we obtain
		\begin{equation*}
		\|u^*\|_{\dot{W}^{s,p}}^p+\|v^*\|_{\dot{W}^{s,p}}^p\leq\int_{B_R(0)}\left(|u^*|^{p^*_s-2\epsilon}+|v^*|^{p^*_s-2\epsilon}+\gamma|u^*|^{\alpha-\epsilon}|v^*|^{\beta-\epsilon}\right)dx.
		\end{equation*}
		Also, note that $\mathcal{J}_\epsilon\left(t_*^{1/p}u^*,t_*^{1/p}v^*\right)\leq\mathcal{J}_\epsilon(u,v)$ for some $t_*\in(0,1]$ such that $\left(t_*^{1/p}u^*,t_*^{1/p}v^*\right)\in\tilde{\mathcal N}_\epsilon(R)$. Hence we choose a minimizing sequence $\{(u_n,v_n)\}\subset\tilde{\mathcal N}_\epsilon(R)$ of $\tilde{A}_\epsilon$ such that $(u_n,v_n)=(u_n^*,v_n^*)$ for any $n$ and $\mathcal{J}_\epsilon(u_n,v_n)\to\tilde{A}_\epsilon$ as $n\to\infty.$ Thus, we get both the sequence $\{u_n\}$, $\{v_n\}$ are bounded in $W^{s,p}_{0}(B_R(0))$. Since, $W^{s,p}_{0}(B_R(0))$ is a reflexive Banach space, upto a subsequence, $u_n\to u_\epsilon$, $v_n\to v_\epsilon$ weakly in $W^{s,p}_{0}(B_R(0)).$ Moreover, as $W^{s,p}_{0}(B_R(0))\hookrightarrow L^{p^*_s-2\epsilon}(B_R(0))$ is a compact embedding, it follows $u_n\to u_\epsilon$, $v_n\to v_\epsilon$ strongly in $L^{p^*_s-2\epsilon}(B_R(0))$. Thereore,
		
		\begin{multline*}
		\int_{B_R(0)}\left(|u_\epsilon|^{p^*_s-2\epsilon}+|v_\epsilon|^{p^*_s-2\epsilon}+\gamma|u_\epsilon|^{\alpha-\epsilon}|v_\epsilon|^{\beta-\epsilon}\right)dx
		\\
		=\lim\limits_{n\to\infty}\int_{B_R(0)}\left(|u_n|^{p^*_s-2\epsilon}+|v_n|^{p^*_s-2\epsilon}+\gamma|u_n|^{\alpha-\epsilon}|v_n|^{\beta-\epsilon}\right)dx
		\\
		=\frac{ p(p^*_s-2\epsilon)}{p^*_s-2\epsilon-p}\lim\limits_{n\to\infty}\mathcal{J}_\epsilon(u_n,v_n)
		=\frac{ p(p^*_s-2\epsilon)}{p^*_s-2\epsilon-p}\tilde{A}_\epsilon(R)>0,
		\end{multline*}
		this yields that $(u_\epsilon,v_\epsilon)\neq(0,0)$ and also $u_\epsilon,\,v_\epsilon$ are nonnegative radially symmetric decreasing. Using the weak lower semicontinuity property of the norm, we also have
		\begin{equation*}
		\|u_\epsilon\|_{\dot{W}^{s,p}}^p+\|v_\epsilon\|_{\dot{W}^{s,p}}^p
		\leq\lim\limits_{n\to\infty}\left(\|u_n\|_{\dot{W}^{s,p}}^p+\|v_n\|_{\dot{W}^{s,p}}^p\right),
		\end{equation*}
		and therefore
		\begin{equation*}
		\|u_\epsilon\|_{\dot{W}^{s,p}}^p+\|v_\epsilon\|_{\dot{W}^{s,p}}^p\leq\int_{B_R(0)}\left(|u_\epsilon|^{p^*_s-2\epsilon}+|v_\epsilon|^{p^*_s-2\epsilon}+\gamma|u_\epsilon|^{\alpha-\epsilon}|v_\epsilon|^{\beta-\epsilon}\right)dx.
		\end{equation*}
		Therefore there exists $t_\epsilon\in(0,1]$ such that $\left(t_\epsilon^{1/p}u_\epsilon,t_\epsilon^{1/p}v_\epsilon\right)\in\tilde{\mathcal N}_\epsilon$  and hence
		\begin{multline*}
		\tilde{A}_\epsilon(R)
		\leq\mathcal{J}_\epsilon\left(t_\epsilon^{1/p}u_\epsilon,t_\epsilon^{1/p}v_\epsilon\right)
		=\frac{ t_\epsilon(p^*_s-2\epsilon-p)}{p(p^*_s-2\epsilon)}\left(\|u_\epsilon\|_{\dot{W}^{s,p}}^p+\|v_\epsilon\|_{\dot{W}^{s,p}}^p\right)
		\\
		\leq\frac{p^*_s-2\epsilon-p}{p(p^*_s-2\epsilon)}\lim\limits_{n\to\infty}\left(\|u_n\|_{\dot{W}^{s,p}}^p+\|v_n\|_{\dot{W}^{s,p}}^p\right)
		=\lim\limits_{n\to\infty}\mathcal{J}_\epsilon(u_n,v_n)=\tilde{A}_\epsilon(R),
		\end{multline*}
		which yields that $t_\epsilon=1$, $\left(u_\epsilon,v_\epsilon\right)\in\tilde{\mathcal N}_\epsilon(R)$, $\tilde{A}_\epsilon(R)=\mathcal{J}_\epsilon(u_\epsilon,v_\epsilon)$ and
		\begin{equation*}
		\|u_\epsilon\|_{\dot{W}^{s,p}}^p+\|v_\epsilon\|_{\dot{W}^{s,p}}^p
		=\lim\limits_{n\to\infty}\left(\|u_n\|_{\dot{W}^{s,p}}^p+\|v_n\|_{\dot{W}^{s,p}}^p\right).
		\end{equation*}
		This proved that $u_n\to u_\epsilon$, $v_n\to v_\epsilon$ strongly in $W^{s,p}_{0}(B_R(0))$. Now by Lagrange multiplier theorem, there exists $\lambda\in\re$ such that
		$$
		\mathcal{J}_\epsilon^{\prime}(u_\epsilon,v_\epsilon)+\lambda G_\epsilon^{\prime}(u_\epsilon,v_\epsilon)=0.
		$$
		Again since $\mathcal{J}_\epsilon^{\prime}(u_\epsilon,v_\epsilon)=G_\epsilon(u_\epsilon,v_\epsilon)=0$ and
		$$
		G_\epsilon^{\prime}(u_\epsilon,v_\epsilon)(u_\epsilon,v_\epsilon)=-(p^*_s-2\epsilon-p)\int_{B_R(0)}\left(|u_\epsilon|^{p^*_s-2\epsilon}+|v_\epsilon|^{p^*_s-2\epsilon}+\gamma|u_\epsilon|^{\alpha-\epsilon}|v_\epsilon|^{\beta-\epsilon}\right)dx<0,
		$$
		we get $\lambda=0$ and hence $\mathcal{J}_\epsilon^{\prime}(u_\epsilon,v_\epsilon)=0$. Since $\tilde{A}_\epsilon(R)=\mathcal{J}_\epsilon(u_\epsilon,v_\epsilon)$ and by Lemma \ref{upper bdd of tilde A epsilon}, we have $u_\epsilon,v_\epsilon\not\equiv0.$ By maximum principle (see, \cite[Lemma~3.3]{DeQu}) we conclude the desired result.
	\end{proof}
	
	\begin{lemma} \label{Lemma 3.8}
		For any $(u,v) \in \tilde{\mathcal N}$, there is a sequence $(u_n,v_n) \in \tilde{\mathcal N} \cap (C^\infty_0(\rn) \times C^\infty_0(\rn))$ such that $(u_n,v_n) \to (u,v)$ in $X$ as $n\to\infty$.
	\end{lemma}
	
	\begin{proof}
		By density, there is a sequence $(\tilde{u}_n,\tilde{v}_n) \in C^\infty_0(\rn) \times C^\infty_0(\rn)$ such that $(\tilde{u}_n,\tilde{v}_n) \to (u,v)$ in $X$ as $n\to\infty$. Let
		\[
		t_n = \left(\frac{\|\tilde u_n\|_{\dot{W}^{s,p}}^p + \|\tilde v_n\|_{\dot{W}^{s,p}}^p}{\displaystyle \int_{\rn} \left(|\tilde{u}_n|^{p^*_s} + |\tilde{v}_n|^{p^*_s} + \gamma\, |\tilde{u}_n|^{\alpha}\, |\tilde{v}_n|^{\beta}\right) dx}\right)^{1/(p_s^* - p)}
		\]
		and note that $t_n \to 1$ since $(u,v) \in \tilde{\mathcal N}$. Then $(u_n,v_n) = (t_n \tilde{u}_n,t_n \tilde{v}_n) \in \tilde{\mathcal N} \cap (C^\infty_0(\rn) \times C^\infty_0(\rn))$ and $(u_n,v_n) \to (u,v)$ in $X$.
	\end{proof}
	
	\begin{lemma} \label{Lemma 3.9}
		There is a minimizing sequence $(u_n,v_n) \in \tilde{\mathcal N} \cap (C^\infty_0(\rn) \times C^\infty_0(\rn))$ for $\tilde{A}$.
	\end{lemma}
	
	\begin{proof}
		Let $(\tilde{u}_n,\tilde{v}_n) \in \tilde{\mathcal N}$ be a minimizing sequence for $\tilde{A}$, i.e., $\mathcal{J}(\tilde{u}_n,\tilde{v}_n) \to \tilde{A}$. By the continuity of $\mathcal{J}$ and Lemma \ref{Lemma 3.8}, there is a $(u_n,v_n) \in \tilde{\mathcal N} \cap (C^\infty_0(\rn) \times C^\infty_0(\rn))$ such that
		\[
		|\mathcal{J}(u_n,v_n) - \mathcal{J}(\tilde{u}_n,\tilde{v}_n)| < \frac{1}{n}.
		\]
		Then $\mathcal{J}(u_n,v_n) \to \tilde{A}$, so $(u_n,v_n) \in \tilde{\mathcal N} \cap (C^\infty_0(\rn) \times C^\infty_0(\rn))$ is a minimizing sequence for $\tilde{A}$.
	\end{proof}
	
	\medskip
	
	\noindent{\bf Proof of Theorem~\ref{in ball}}: First, we prove that
	\begin{equation}\label{A_R A are same}
	\tilde{A}(R)=\tilde{A}\,\text{ for every }R>0.
	\end{equation}
	Let $R_1<R_2$, then $\tilde{\mathcal N}(R_1)\subset\tilde{\mathcal N}(R_2)$ and hence by definition we have $\tilde{A}(R_2)\leq\tilde{A}(R_1).$ To prove reverse inequality, let $(u,v)\in\tilde{\mathcal N}(R_2)$ and define
	$$
	\left(u_1(x),v_1(x)\right):=\left(\frac{R_2}{R_1}\right)^{\frac{ N-sp}{p}}\left(u\left(\frac{R_2}{R_1}x\right),\,v\left(\frac{R_2}{R_1}x \right)\right).
	$$
	Clearly, $(u_1,v_1)\in\tilde{\mathcal N}(R_1)$. Therefore, we get
	$$
	\tilde{A}(R_1)\leq\mathcal{J}(u_1,v_1)=\mathcal{J}(u,v),\,\text{ for any }(u,v)\in\tilde{\mathcal N}(R_2),
	$$
	and this implies that $\tilde{A}(R_1)\leq\tilde{A}(R_2).$ So, we obtain $\tilde{A}(R_1)=\tilde{A}(R_2)$. Let $(u_n,v_n)\in\tilde{\mathcal N}$ be a minimizing sequence of $\tilde{A}$. In view of Lemma \ref{Lemma 3.9}, we may assume that $u_n,v_n\in W^{s,p}_{0}(B_{R_n}(0)) $ for some $R_n>0.$ Then, $(u_n,v_n)\in\tilde{\mathcal N}(R_n)$ and
	$$
	\tilde{A}=\lim\limits_{n\to\infty}\mathcal{J}(u_n,v_n)\geq\lim\limits_{n\to\infty}\tilde{A}(R_n)=\tilde{A}(R),
	$$
	and hence \eqref{A_R A are same} holds.
	
	Let  $(u,v)\in\tilde{\mathcal N}(R)$ be arbitrary, then there exists $t_\epsilon>0$ with $t_\epsilon\to1$ as $\epsilon\to0$ such that $\left(t_\epsilon^{1/p}u,\,t_\epsilon^{1/p}v\right)\in\tilde{\mathcal N}_\epsilon(R)$. Therefore, we have
	$$
	\limsup\limits_{\epsilon\to0}\tilde{A}_\epsilon\leq\limsup\limits_{\epsilon\to0}\mathcal{J}_\epsilon\left(t_\epsilon^{1/p}u,\,t_\epsilon^{1/p}v\right)=\mathcal{J}(u,v).
	$$
	Thus, using \eqref{A_R A are same} we obtain
	\begin{equation}\label{upper bdd of limsup A_epsilon}
	\limsup\limits_{\epsilon\to0}\tilde{A}_\epsilon(R)\leq\tilde{A}(R)=\tilde{A}.
	\end{equation}
	By Proposition \ref{radially symmetric solution}, let $(u_\epsilon,v_\epsilon)$ be a positive least energy solution of \eqref{epsilon system}, which is radially symmetric nonincreasing. Then by Lemma \ref{Lemma 3.5}, for any $\epsilon_0 \in (0,\min\{\alpha - 1,\beta - 1,(p_s^* - p)/2\})$, there exists a constant $C_{\epsilon_0} > 0$ such that
	\begin{equation} \label{A_epsilon lowerbdd indp of epsilon}
	\tilde{A}_\epsilon(R) = \frac{p_s^* - p - 2 \epsilon}{p\, (p_s^* - 2 \epsilon)} \left(\|u_\epsilon\|_{\dot{W}^{s,p}}^p + \|v_\epsilon\|_{\dot{W}^{s,p}}^p\right) \ge C_{\epsilon_0} \quad \forall \epsilon \in (0,\epsilon_0].
	\end{equation}
	Therefore, from \eqref{upper bdd of limsup A_epsilon} we get $u_\epsilon,\,v_\epsilon\in W^{s,p}_{0}(B_{R}(0)$ are uniformly bounded. Thus, by reflexivity upto a subsequence $u_\epsilon\to u_0$ and $v_\epsilon\to v_0$ weakly in $W^{s,p}_{0}(B_{R}(0))$ as $\epsilon\to0.$ Since \eqref{epsilon system} is a subcritical system in bounded domain, passing the limit $\epsilon\to 0$, it follows that $(u_0,v_0)$ is a solution of the following system
	\begin{equation*}
	\begin{cases}
	(-\Delta_p)^s u =|u|^{p^*_s-2}u+ \frac{\alpha\gamma}{p_s^*}|u|^{\alpha-2}u|v|^{\beta}\;\;\text{in}\;B_R(0),\vspace{.2cm}
	\\
	(-\Delta_p)^s v =|v|^{p^*_s-2}v+ \frac{\beta\gamma}{p_s^*}|v|^{\beta-2}v|u|^{\alpha}\;\;\text{in}\;B_R(0),
	\\
	u,\;v\in W^{s,p}_{0}(B_R(0)).
	\end{cases}
	\end{equation*}
	Also note that $u_0,\,v_0$ are nonnegative and from \eqref{A_epsilon lowerbdd indp of epsilon} we see that $(u_0,v_0)\neq(0,0)$.  We may now assume that $u_0\not\equiv0$. Therefore, by strong maxmimum principle (see, \cite{DeQu}) we obtain $u_0>0$ in $B_R(0).$ Further, we claim, $v_0\not\equiv 0$. If $v_0\equiv 0$ then substituting $(u_0, v_0)$ in the above system of equation leads $u_0$ is a positive solution to  $(-\Delta_p)^s u =|u|^{p^*_s-2}u$ in $B_R(0)$. Since $u_0\in W^{s,p}_0(B_R(0))$, it follows
	\begin{equation}\label{u0-1}\mathcal{J}(u_0, 0)=\frac{1}{p}\|u_0\|_{\dot{W}^{s,p}}^p-\frac{1}{p^*_s}\int_{\rn}u_0^{p^*_s}dx=\frac{1}{p}\|u_0\|_{\dot{W}^{s,p}}^p-\frac{1}{p^*_s}\int_{B_R(0)}u_0^{p^*_s}dx=\frac{s}{N}\|u_0\|_{\dot{W}^{s,p}}^p.
	\end{equation}
	We also observe that $(u_0,0)$, $(0,u_0)\in\mathcal{\tilde N}$. Therefore, using \eqref{tl-A-U}, we have
	\begin{equation}\label{tlA-u0}
	\tilde A<\min\{\inf_{(u,0)\in\mathcal{\tilde N}}\mathcal{J}(u,0), \inf_{(0,v)\in\mathcal{\tilde N}}\mathcal{J}(0,v) \}\leq \min\{\mathcal{J}(u_0,0), \mathcal{J}(0,u_0)  \}=\mathcal{J}(u_0,0).
	\end{equation}
	Combining \eqref{u0-1} and \eqref{tlA-u0} together yields
	\begin{equation}\label{tlA-u01}
	\tilde A<\frac{s}{N}\|u_0\|_{\dot{W}^{s,p}}^p.
	\end{equation}
	Further, by \eqref{upper bdd of limsup A_epsilon} and the fact that $(u_\epsilon, v_\epsilon)$ is a positive least energy solution of \eqref{epsilon system}, it follows
	\begin{multline*}
	\tilde A \geq \limsup_{\epsilon\to 0}\tilde A_\epsilon(R)
	=\limsup_{\epsilon\to 0} \mathcal{J}_\epsilon(u_\epsilon, v_\epsilon)\\
	=\limsup_{\epsilon\to 0}\left[\frac{1}{p}\|u_\epsilon\|_{\dot{W}^{s,p}}^p+\|v_\epsilon\|_{\dot{W}^{s,p}}^p)-\frac{1}{p^*_{s}-2\epsilon}\int_{B_R(0)}(u_\epsilon^{p^*_{s}-2\epsilon}+v_\epsilon^{p^*_{s}-2\epsilon}+\gamma u_{\epsilon}^{\alpha-\epsilon}v_\epsilon^{\beta-\epsilon})dx \right]\\
	=\limsup_{\epsilon\to 0}\left(\frac{1}{p}-\frac{1}{p^*_{s}-2\epsilon}\right)\left(\|u_\epsilon\|_{\dot{W}^{s,p}}^p+\|v_\epsilon\|_{\dot{W}^{s,p}}^p\right)
	\geq \frac{s}{N}(\|u_0\|_{\dot{W}^{s,p}}^p+\|v_0\|_{\dot{W}^{s,p}}^p)\\
	=\frac{s}{N}\|u_0\|_{\dot{W}^{s,p}}^p
	>\tilde A \quad \text{(by \eqref{tlA-u01} )},
	\end{multline*}
	which is a contradiction. Hence $v_0\not\equiv 0$ and again by strong maximum principle we obtain $v_0>0$ in $B_R(0).$ Moreover, as $(u_\epsilon, v_\epsilon)$ is radial and $u_\epsilon\to u$ a.e. and $v_\epsilon \to v$ a.e. (up to a subsequence), we also have $u_0, v_0$ are radial functions.  Hence $(u_0, v_0)$ is a positive radial solution to \eqref{system_ball}.
	\hfill$\square$

	\medskip

	\noindent{\bf Proof of Theorem~\ref{gamma near 0}}: To prove the existence of $(k(\gamma),\ell(\gamma))$ for small $\gamma>0$, recalling \eqref{algebraic eq3},  we denote $F_i(k,\ell,\gamma)$ instead of $F_i(k,\ell)$, $i=1,2$ in this case. Let $k(0)=1=\ell(0)$, then $F_i(k(0),\ell(0),0)=0$, $i=1,2.$
	Clearly, we have
	\begin{equation*}
	\frac{\partial F_1}{\partial k}(k(0),\ell(0),0)=\frac{\partial F_2}{\partial\ell}(k(0),\ell(0),0)=\frac{p^*_s-p}{p}>0
	\end{equation*}
	and
	$$
	\frac{\partial F_1}{\partial\ell}(k(0),\ell(0),0)=\frac{\partial F_2}{\partial k}(k(0),\ell(0),0)=0.
	$$
	Therefore, the jacobian determinant is $J_{F}(k(0),\ell(0))=\frac{(p^*_s-p)^2}{p^2}>0$, where $F:=(F_1,F_2).$
	Therefore, by the implicit function theorem,  $k(\gamma),\,\ell(\gamma)$ are well defined functions and of class $C^1$ in $(-\gamma_2,\gamma_2)$ for some $\gamma_2>0$ and $F_i(k,\ell,\gamma)=0$ for $\gamma\in(-\gamma_2,\gamma_2)$. Then $\left(k(\gamma)^{1/p}U,\ell(\gamma)^{1/p}U\right)$ is a positive solution of \eqref{MAT2}. Since $\lim\limits_{\gamma\to0}(k(\gamma)+\ell(\gamma))=2$. Thus there exists $\gamma_1\in(0,\gamma_2]$ such that $k(\gamma)+\ell(\gamma)>1$ for all $\gamma\in(0,\gamma_1)$.
	Therefore, by  \eqref{tl-A-U} we get
	$$
	\mathcal{J}\left(k(\gamma)^{1/p}U,\ell(\gamma)^{1/p}U\right)=\frac{s}{N}(k(\gamma)+\ell(\gamma))\mathcal{S}^{\frac{N}{sp}}>\frac{s}{N}\,\mathcal{S}^{\frac{N}{sp}}=\tilde{A} .
	$$
	This completes the proof.
	\hfill$\square$

	\vspace{10mm}
	
	{\bf Acknowledgement:} The research of M. Bhakta is partially supported by the SERB WEA grant (WEA/2020/000005) and DST Swarnajaynti fellowship (SB/SJF/2021-22/09). K. Perera is partially supported by the Simons Foundation grant 962241. F. Sk is partially supported by the SERB grant WEA/2020/000005.
	
	\medskip
	
	{\bf Competing interests.} The authors have no competing interests to declare that are relevant to the content of this article.

\end{document}